%% file: DRvol_main.tex
\documentclass{amsart}

\usepackage[english]{babel}
\usepackage{csquotes}
\usepackage[style=alphabetic,
            isbn=false,
            doi=false,
            url=false,
            backend=bibtex,
            maxnames=5,
            maxalphanames = 5,
            giveninits=true
]{biblatex}
\AtEveryBibitem{\clearlist{language}}
\bibliography{sources}
\usepackage{amsmath,amsfonts,amsthm,mathtools, amssymb}
\usepackage{xcolor}
\usepackage{tikz-cd}
\usepackage{hyperref}
\usepackage{aliascnt}
\usepackage{xparse}
\usepackage{tensor}
\usepackage{caption}

\usepackage{tabto}
\usepackage{imakeidx}

 \RequirePackage{filecontents}
\indexsetup{level=\subsection*,toclevel=\paragraph,headers={ \MakeUppercase{Euler sequence on strata}}{\MakeUppercase{Euler sequence on strata}}, noclearpage,firstpagestyle=headings}
\makeindex[name=graph,title=Graphs and levels,columns=1,columnseprule,options=-s Idx.ist]

\usepackage{tikz}
\usetikzlibrary{matrix,arrows,calc}
\usetikzlibrary{positioning}
\usetikzlibrary{decorations.pathreplacing,decorations.markings,decorations.pathmorphing}
\usetikzlibrary{positioning,arrows,patterns}
\usetikzlibrary{cd}
\usetikzlibrary{intersections}

\tikzset{
	every loop/.style={very thick},
	comp/.style={circle,fill,black,,inner sep=0pt,minimum size=5pt},
	order bottom left/.style={pos=.05,left,font=\tiny},
	order top left/.style={pos=.9,left,font=\tiny},
	order bottom right/.style={pos=.05,right,font=\tiny},
	order top right/.style={pos=.9,right,font=\tiny},
	order node dis/.style={text width=.75cm},
	circled number/.style={circle, draw, inner sep=0pt, minimum size=12pt},
	below left with distance/.style={below left,text height=10pt},
    below right with distance/.style={below right,text height=10pt}
	}

\makeatletter
\ifcsname phantomsection\endcsname
    \newcommand*{\@gobblenexttocentry}[9]{}
\else
    \newcommand*{\@gobblenexttocentry}[4]{}
\fi
\newcommand*{\addsubsection}{%
    \addtocontents{toc}{\protect\@gobblenexttocentry}%
    \subsection*}
\makeatother

\begin{document}

\def\subsectionautorefname{Section}
\def\subsubsectionautorefname{Section}
\def\sectionautorefname{Section}
\def\equationautorefname~#1\null{(#1)\null}

\include{macros}


\title{Completed volumes and the DR-cycle}

\begin{abstract}
  We show that the completed volumes introduced by Duriev-Goujard-Yakovlev as an approximation to compute Masur-Veech volumes via Witten-Kontsevich's combinatorial classes agrees with the top intersection of the tautological class on the double ramification cycle, computable as a coefficient of a Chiodo class.
\par
For the proof we describe the components of the double ramification cycle and their excess intersection classes to the extent seen by the top tautological intersection. This gives a recursion computing completed volumes in terms of volumes appearing in a certain set of level graphs, not only for quadratic differentials. It also completes the work of Duriev-Goujard-Yakovlev solving the technically most involved case of strata with two singularities.
  \end{abstract}

\author{Martin M\"oller}
\thanks{Research is supported  
  by the DFG-project MO 1884/3-1 and the Collaborative Research Centre
TRR 326 ``Geometry and Arithmetic of Uniformized Structures.''}
\address{
Institut f\"ur Mathematik, Goethe--Universit\"at Frankfurt,
Robert-Mayer-Str. 6--8,
60325 Frankfurt am Main, Germany
}
\email{moeller@math.uni-frankfurt.de}

\author{Miguel Prado}
\email{prado@math.uni-frankfurt.de}

      \maketitle

\tableofcontents

\input{sec_intro}


\input{sec_background}

\input{sec_support}



\input{sec_normalsheaf}

\input{sec_excess}
\input{sec_volumes}

\input{sec_examples}

\printbibliography

\end{document}

%% file: macros.tex
\newcommand{\mynewtheorem}[4]{
  \if\relax\detokenize{#3}\relax 
    \if\relax\detokenize{#4}\relax 
      \newtheorem{#1}{#2}
    \else
      \newtheorem{#1}{#2}[#4]
    \fi
  \else
    \newaliascnt{#1}{#3}
    \newtheorem{#1}[#1]{#2}
    \aliascntresetthe{#1}
  \fi
  \expandafter\def\csname #1autorefname\endcsname{#2}
}

\mynewtheorem{theorem}{Theorem}{}{section}
\mynewtheorem{lemma}{Lemma}{theorem}{}
\mynewtheorem{rem}{Remark}{lemma}{}
\mynewtheorem{prop}{Proposition}{lemma}{}
\mynewtheorem{cor}{Corollary}{lemma}{}
\mynewtheorem{definition}{Definition}{lemma}{}
\mynewtheorem{question}{Question}{lemma}{}
\mynewtheorem{assumption}{Assumption}{lemma}{}
\mynewtheorem{example}{Example}{lemma}{}


\def\defbb#1{\expandafter\def\csname b#1\endcsname{\mathbb{#1}}}
\def\defcal#1{\expandafter\def\csname c#1\endcsname{\mathcal{#1}}}
\def\deffrak#1{\expandafter\def\csname frak#1\endcsname{\mathfrak{#1}}}
\def\defop#1{\expandafter\def\csname#1\endcsname{\operatorname{#1}}}
\def\defbf#1{\expandafter\def\csname b#1\endcsname{\mathbf{#1}}}

\makeatletter
\def\defcals#1{\@defcals#1\@nil}
\def\@defcals#1{\ifx#1\@nil\else\defcal{#1}\expandafter\@defcals\fi}
\def\deffraks#1{\@deffraks#1\@nil}
\def\@deffraks#1{\ifx#1\@nil\else\deffrak{#1}\expandafter\@deffraks\fi}
\def\defbbs#1{\@defbbs#1\@nil}
\def\@defbbs#1{\ifx#1\@nil\else\defbb{#1}\expandafter\@defbbs\fi}
\def\defbfs#1{\@defbfs#1\@nil}
\def\@defbfs#1{\ifx#1\@nil\else\defbf{#1}\expandafter\@defbfs\fi}
\def\defops#1{\@defops#1,\@nil}
\def\@defops#1,#2\@nil{\if\relax#1\relax\else\defop{#1}\fi\if\relax#2\relax\else\expandafter\@defops#2\@nil\fi}
\makeatother

\defbbs{ZHQCNPALRVWG}
\defcals{ABOPQMNXYLTRAEHZKCFIL}
\deffraks{apijklmnopqueRC}
\defops{IVC, PGL,SL,mod,Spec,Re,Gal,Tr,End,GL,Hom,PSL,H,div,Aut,rk,Mod,R,T,Tr,Mat,Vol,MV,Res,vol,Z,diag,Hyp,hyp,hl,ord,Im,ev,U,dev,c,CH,fin,pr,Pic,lcm,ch,td,LG,id,Sym,Aut,Log,tw,Ch,vp,stat}
\defbfs{uvzwp} 

\def\ep{\varepsilon}
\def\ve{\varepsilon}
\def\abs#1{\lvert#1\rvert}
\def\dd{\mathrm{d}}
\def\WP{\mathrm{WP}}
\def\inj{\hookrightarrow}
\def\eq{=}

\def\i{\mathrm{i}}
\def\e{\mathrm{e}}
\def\st{\mathrm{st}}
\def\ct{\mathrm{ct}}
\def\ab{\mathrm{ab}}
\def\na{\mathrm{na}}
\def\red{\mathrm{red}}
\def\SStar{\mathrm{SStar}}
\def\SF{\mathrm{SF}}
\def\SSStar{\mathrm{SSStar}}

\def\uC{\underline{\bC}}
\def\ol{\overline}
  
\def\Vrel{\bV^{\mathrm{rel}}}
\def\Wrel{\bW^{\mathrm{rel}}}
\def\twolev{\mathrm{LG_1(B)}}

\def\be{\begin{equation}}   \def\ee{\end{equation}}     \def\bes{\begin{equation*}}    \def\ees{\end{equation*}}
\def\ba{\be\begin{aligned}} \def\ea{\end{aligned}\ee}   \def\bas{\bes\begin{aligned}}  \def\eas{\end{aligned}\ees}
\def\={\;=\;}  \def\+{\,+\,} \def\m{\,-\,}

\newcommand*{\proj}{\mathbb{P}}
\newcommand{\barmoduli}[1][g]{{\overline{\mathcal M}}_{#1}}
\newcommand{\moduli}[1][g]{{\mathcal M}_{#1}}
\newcommand{\omoduli}[1][g]{{\Omega\mathcal M}_{#1}}
\newcommand{\komoduli}[1][g]{{\Omega\mathcal M}^{\otimes k}_{#1}}
\newcommand{\modulin}[1][g,n]{{\mathcal M}_{#1}}
\newcommand{\omodulin}[1][g,n]{{\Omega\mathcal M}_{#1}}
\newcommand{\zomoduli}[1][]{{\mathcal H}_{#1}}
\newcommand{\barzomoduli}[1][]{{\overline{\mathcal H}_{#1}}}
\newcommand{\pomoduli}[1][g]{{\proj\Omega\mathcal M}_{#1}}
\newcommand{\pomodulin}[1][g,n]{{\proj\Omega\mathcal M}_{#1}}
\newcommand{\pobarmoduli}[1][g]{{\proj\Omega\overline{\mathcal M}}_{#1}}
\newcommand{\pobarmodulin}[1][g,n]{{\proj\Omega\overline{\mathcal M}}_{#1}}
\newcommand{\potmoduli}[1][g]{\proj\Omega\tilde{\mathcal{M}}_{#1}}
\newcommand{\obarmoduli}[1][g]{{\Omega\overline{\mathcal M}}_{#1}}
\newcommand{\obarmodulio}[1][g]{{\Omega\overline{\mathcal M}}_{#1}^{0}}
\newcommand{\otmoduli}[1][g]{\Omega\tilde{\mathcal{M}}_{#1}}
\newcommand{\pom}[1][g]{\proj\Omega{\mathcal M}_{#1}}
\newcommand{\pobarm}[1][g]{\proj\Omega\overline{\mathcal M}_{#1}}
\newcommand{\pobarmn}[1][g,n]{\proj\Omega\overline{\mathcal M}_{#1}}
\newcommand{\princbound}{\partial\mathcal{H}}
\newcommand{\omoduliinc}[2][g,n]{{\Omega\mathcal M}_{#1}^{{\rm inc}}(#2)}
\newcommand{\obarmoduliinc}[2][g,n]{{\Omega\overline{\mathcal M}}_{#1}^{{\rm inc}}(#2)}
\newcommand{\pobarmoduliinc}[2][g,n]{{\proj\Omega\overline{\mathcal M}}_{#1}^{{\rm inc}}(#2)}
\newcommand{\otildemoduliinc}[2][g,n]{{\Omega\widetilde{\mathcal M}}_{#1}^{{\rm inc}}(#2)}
\newcommand{\potildemoduliinc}[2][g,n]{{\proj\Omega\widetilde{\mathcal M}}_{#1}^{{\rm inc}}(#2)}
\newcommand{\omoduliincp}[2][g,\lbrace n \rbrace]{{\Omega\mathcal M}_{#1}^{{\rm inc}}(#2)}
\newcommand{\obarmoduliincp}[2][g,\lbrace n \rbrace]{{\Omega\overline{\mathcal M}}_{#1}^{{\rm inc}}(#2)}
\newcommand{\obarmodulin}[1][g,n]{{\Omega\overline{\mathcal M}}_{#1}}
\newcommand{\LTH}[1][g,n]{{K \overline{\mathcal M}}_{#1}}
\newcommand{\PLS}[1][g,n]{{\bP\Xi \mathcal M}_{#1}}

\DeclareDocumentCommand{\LMS}{ O{\mu} O{g,n} O{}}{\Xi\overline{\mathcal{M}}^{#3}_{#2}(#1)}
\DeclareDocumentCommand{\Romod}{ O{\mu} O{g,n} O{}}{\Omega\mathcal{M}^{#3}_{#2}(#1)}

\newcommand*{\Tw}[1][\Lambda]{\mathrm{Tw}_{#1}}  
\newcommand*{\sTw}[1][\Lambda]{\mathrm{Tw}_{#1}^s}  

\newcommand{\bfa}{{\bf a}}
\newcommand{\bfb}{{\bf b}}
\newcommand{\bfd}{{\bf d}}
\newcommand{\bfe}{{\bf e}}
\newcommand{\bff}{{\bf f}}
\newcommand{\bfg}{{\bf g}}
\newcommand{\bfh}{{\bf h}}
\newcommand{\bfj}{{\bf j}}
\newcommand{\bfm}{{\bf m}}
\newcommand{\bfn}{{\bf n}}
\newcommand{\bfp}{{\bf p}}
\newcommand{\bfq}{{\bf q}}
\newcommand{\bft}{{\bf t}}
\newcommand{\bfP}{{\bf P}}
\newcommand{\bfR}{{\bf R}}
\newcommand{\bfU}{{\bf U}}
\newcommand{\bfu}{{\bf u}}
\newcommand{\bfz}{{\bf z}}

\newcommand{\bfl}{{\boldsymbol{\ell}}}
\newcommand{\bfmu}{{\boldsymbol{\mu}}}
\newcommand{\bfeta}{{\boldsymbol{\eta}}}
\newcommand{\bftau}{{\boldsymbol{\tau}}}
\newcommand{\bfomega}{{\boldsymbol{\omega}}}
\newcommand{\bfsigma}{{\boldsymbol{\sigma}}}

\newcommand{\wh}{\widehat}
\newcommand{\wt}{\widetilde}

\newcommand{\ps}{\mathrm{ps}}  

\newcommand{\tdpm}[1][{\Gamma}]{\mathfrak{W}_{\operatorname{pm}}(#1)}
\newcommand{\tdps}[1][{\Gamma}]{\mathfrak{W}_{\operatorname{ps}}(#1)}

\newlength{\halfbls}\setlength{\halfbls}{.5\baselineskip}

\newcommand*{\Teichmuller}{Teich\-m\"uller\xspace}

\newcommand{\bd}[1]{{\mathbf{#1}}}
\newcommand{\cat}[1]{\bd{#1}}
\DeclareDocumentCommand{\MSmu}{ O{\mu} }{\mathcal{MS}_{#1}}
\DeclareDocumentCommand{\MS}{ O{\mu} }{\mathcal{MS}}
\DeclareDocumentCommand{\GMSmu}{ O{\mu} }{\mathcal{GMS}_{#1}}
\DeclareDocumentCommand{\GMS}{ O{\mu} }{\mathcal{GMS}}
\def\GRC{\mathrm{GRC}}
\def\DRL{\mathsf{DRL}}
\def\DR{\mathsf{DR}}
\def\logDR{\mathsf{logDR}}
\def\rel{\mathrm{rel}}
\newcommand{\trop}{\mathsf{trop}}

\newcommand{\Picabs}{\mathfrak{Pic}}
\newcommand{\Picrel}{\on{Pic}}

\newcommand{\ghost}{\overline{\mathsf{M}}}
\newcommand{\M}{\mathsf{M}}
\newcommand{\ul}[1]{{\underline{#1}}}

\newcommand{\kernel}{\operatorname{ker}}
\newcommand{\coker}{\operatorname{coker}}

%% file: sec_intro.tex
\section{Introduction}
\label{sec:intro}

The moduli space of projectivized $k$-differentials $\bP \komoduli[g,n](\mu)$ with zeros and poles of orders prescribed by $\mu = (m_1,\ldots,m_n)$ comes with a natural finite volume form. In particular for the case of Abelian differentials $k=1$ and quadratic differentials $k=2$ these Masur-Veech volumes of these moduli spaces have important implicitations for the dynamics on translation and half-translation surfaces, see \cite{ZorichFlat, Filip} for surveys and \cite{Nguyenddif} for interpretations for higher~$k$. 
\par
A challenging problem is to find formulas for the computation of these volumes, that can be efficiently evaluated and such that the large genus limits can be derived. For the case of abelian differentials, the large limit conjecture \cite{ADGZZconj} has been proven in recent years \cite{CMSZ, Agg20, SauvagetLarge}. For quadratic differentials the principal strata (all $m_i = \pm 1$) was shown in \cite{Agg21} while the general case is open.
\par
There are three strategies to compute Masur-Veech volumes. First, the initial strategy of counting square-tiled surfaces of Zorich and Eskin-Okounkov \cite{EO01} and relying on the quasimodularity of the generating function  has sucessfully been used for abelian differentials \cite{CMSZ} and for a low genus data base \cite{Goujardexplicit}, but it will play no role in the sequel of this paper. Second, one may decompose the (half-) translation surfaces into ribbon graphs with glued-in cylinders and count ribbon graphs. Key to make this strategy usefuly is to relate the count of integer metrics on ribbon graphs to intersection numbers of Witten-Kontsevich's combinatorial classes \cite{Witten2dgrav,KontAiry} on the moduli space of stable curves~$\barmoduli[g,n]$. This strategy was implemented for abelian differentials \cite{DGZZfreq, DGZZlarge} and principal strata of quadratic differentials \cite{DGZZmeander, DGZZmeanderhigher}. Third, Masur-Veech volumes can be computed as intersection numbers of tautological classes on the multi-scale compactification of the moduli spaces $\bP \komoduli[g,n](\mu)$ themselves, proven for abelian differentials \cite{CMSZ} and quadratic differentials without relative periods  and conjecturally in general \cite{CMSprincipal}. In the case without relative periods the class that computes the Masur-Veech volume is simply the top power $\xi^d$ of the tautological line bundle where $d= \dim \bP \komoduli[g,n](\mu)$.
\par
\medskip
\paragraph{\textbf{Ribbon graph completed volumes}} To adapt the ribbon graph strategy for quadratic differentials with general odd entries $m_i$, Duriev--Goujard-Yakovlev introduced in \cite{DGY} the notion of \emph{completed volumes $\ol{\vol}(\mu)$} of moduli spaces of quadratic differentials: Instead of counting metrics on Ribbon graphs, they evaluate Kontsevich's polynomials in terms of combinatorial classes, even though they overcount by contributions of degenerate ribbon graphs. Their main theorem \cite[Theorem~2.6]{DGY} is (for the case of $n \geq 4$ singularities) a recursive formula for completed volume in terms of the Masur-Veech volume and lower complexity volumes.
\par
\par
\medskip
\paragraph{\textbf{Double ramification (DR)-cycle completed volumes}}
To evaluate intersection numbers on the multi-scale compactification one may also consider using a completion. The double ramification (DR)-cycle can be considered as an alternative 'compactification' of $\bP \komoduli[g,n](\mu)$ in an appropriate blowup of $\barmoduli[g,n]$ having additional components that lie entirely over the boundary of $\barmoduli[g,n]$. Chiodo--Holmes \cite{ChiodoHolmes} recently proved a closed formula for intersection numbers of arbitrary powers~$\xi^r$ of the tautological line bundle on the DR-cycle, which was conjectured in~\cite{To2}. It is thus natural to define the \emph{DR-completed volume $\ol{\vol}^\DR(\mu)$} of any of the moduli spaces $\bP \komoduli[g,n](\mu)$ to be the appropriate intersection number, that is, $\xi^d$ in the case without relative periods, against the DR-cycle, and to aim for a recursive formula in terms of the true Masur-Veech volume and of volumes of lower complexity moduli spaces.
\par
Refering to Sections~\ref{sec:excess} and~\ref{sec:completedvolumes} for the precise volume definitions we can state our first main result simply as:
\par
\begin{theorem} \label{thm:intromain}
The DR-completed volume and the (Ribbon graph) completed volume agree, i.e.,
\bes
\ol{\vol}^\DR(\mu) \= \ol{\vol}(\mu) 
\ees
for all signatures of quadratic differentials with $m_i>-2$ odd for $i=1,\ldots,n$.
\end{theorem}
\par
In fact, the theorem is not just and equality of two numbers, but a termwise equality of two recursions that form the two constituents of the proof.
\par
The theorem now opens a two-fold way to approach the large genus asymptotic conjectures: Use either of the two tautological class expression for the completed volume to compute its large genus limit and deduce the conjecture in \cite{ADGZZconj} by inverting the recusion~\eqref{eq:DRrecintro} below. We give a comparison of the two tautological class expression at the end of the paper in Section~\ref{sec:examples}
\par
\medskip
\paragraph{\textbf{Components of the DR-cycle and the excess bundle}} The space of generalized multi-scale differentials $\GMS_\mu$ (also known as DR-locus) is the locus inside the rubber moduli space $\bP\cat{Rub}$ where the line bundle~$\cL_\mu$ defined as the $k$-th power of the canonical twisted by the marked point with weight~$-\mu$ becomes trivial. Here $\bP\cat{Rub}$ is the stack of (projectivized) rubber curves introduced by Marcus and Wise \cite{MarcusWiselog}. One of the components of~$\GMS_\mu$ is the compactification $\MS_k(\mu)$ of $\bP \komoduli[g,n](\mu)$ by multi-scale differentials \cite{LMS,To2}. The DR-cycle is the Gysin pullback of the section of the relative Picard variety over $\bP \cat{Rub}$ defined by~$\cL_\mu$ along the zero section. It is a class in $\CH^g(\bP\cat{Rub})$ supported on $\GMS_\mu$, see Section~\ref{sec:genMS} for details.
\par
The components of the DR-cycle whose image in $\barmoduli[g,n]$ is full-dimensional are called star-shaped graphs and are used in many places, e.g.\ \cite{FP18, HolmesSchmitt,BHPPS23} for a selection of papers. We complement this by a full structure result of~$\GMS_\mu$:
\begin{itemize}
\item The irreducible components $\GMS_\mu(\Gamma)$ of $\GMS_\mu$ are classified in Proposition~\ref{prop:GMSsubstack}, using level graphs~$\Gamma$ that \emph{impose a global residue condition (GRC) at every level} (Definition~\ref{def:GRCimpose})
\item The local equations cutting out $\GMS_\mu$ in $\bP\cat{Rub}$ are products of level parameters and function that specialize to GRCs on $\MS_k(\mu)$, see Proposition~\ref{prop:GMSequations} for details.
\end{itemize}
\par
Our main contribution here is the computation DR-cycle in Proposition~\ref{prop:DRwithxitop} as a sum over components of~$\GMS_\mu$ with the corresponding excess bundles to the extent seen by intersection with the top power~$\xi^d$ of the tautological bundle. This combines that:
\begin{itemize}
\item The intersection with $\xi^d$ is non-zero only for two types of level graphs~$\Gamma$, \emph{special star-shaped graphs} and \emph{sun-flower graphs}, as defined in Definition~\ref{def:SSStar} and~\ref{def:sunflower}, see Proposition~\ref{prop:xitopgraphs}.
\item The Chern polynomials of the normal bundle of $\GMS_\mu(\Gamma)$ in $\bP\cat{Rub}$ is computed in Theorem~\ref{thm:nbGMSapprox} up to terms unseen by $\xi^d$.
\end{itemize}
\par
\par
The reader interested in the full expression of the DR-cycle as a sum over components of the conormal sheaf e.g.\ for intersection products with other products powers of the taulogical bundle times tautological classes is invited to complete Proposition~\ref{prop:tangMS} with the local computation to derive the contribution supported on GRC-divisor and to add the contributions from intersections of $\GMS_\mu$-components using Proposition~\ref{prop:GMSequations}.
\par
We conclude in Theorem~\ref{thm:DRcompletedvolgraphs}: The DR-completed volume can be written as a sum
\be \label{eq:DRrecintro}
\ol{\vol}^\DR(\mu) \=  {\vol}(\mu) + \sum_{\Gamma \in \SF(\mu) \cup \SSStar(\mu)} \frac{1}{\Aut(\Gamma)k^{h_\ab}} \prod_{e \in E(\Gamma)} \kappa_e \prod_{v \in V(\Gamma)} \vol(\mu_v)
\ee
over all special simple star graphs and sunflower graphs, where $\mu_v$ is the signature of all legs adjacent to the vertex~$v$, and $h_\ab$ is the number of holomorphic abelian top components of the corresponding graph $\Gamma$. Here~$\kappa_e$ are the enhancements (also known as prongs or as twists) associated with the edges of~$\Gamma$.
\par
\medskip
\paragraph{\textbf{Ribbon graph counting and the two singularity case}} To prove
Theorem~\ref{thm:intromain} for $n \geq 4$ we identify the terms in the main Theorem~2.6 of \cite{DGY} with the recursion in~\eqref{eq:DRrecintro}. This is done in Theorem~\ref{thm:completedvolgraphs3plus}. In this case the special star shaped graphs cannot occur. Note that a 'volume' appears in~\eqref{eq:DRrecintro} for all vertices $v \in V(\Gamma)$, not just those on top level, which correspond to moduli spaces of finite volume differentials. For those on lower level, corresponding to infinite volume (meromorphic) differentials, the volume is to be interpreted as intersection number of the top power of the tautological bundle~$\xi^d$ -- and this can be computed to give the coefficients \cite{DGY}.
\par
For $n = 2$ we use the hindsight of the intended completed volume formula as graph sum to solve what appeared in \cite{DGY} the 'cumbersome' (and thus excluded) case:
\par
The overcounting of the Kontsevich polynomial occurs, if the ribbon graph admits
integer metrics degenerating to metrics assigning length zero to some edges and this edges are \emph{static}, i.e. the ribbon becomes bipartite after the edge has been removed. To each such ribbon we can associate a level graph (see Section~\ref{sec:ribbon} for the precise procedure) and the sunflower graphs are precisely the level graphs associated to ribbon graphs with a non-separating static edge.
\par
We thus need two steps: First, to analyze how the pieces of the ribbon graphs can be pieced together to give a valid configuration of static edges. This is the content of Proposition~\ref{prop:Gstat}. Second, we need to count in how many ways those pieces can be assembled to form a valid ribbon graph. It is through this count that the 'volumes' of the lower level differentials occur in Proposition~\ref{prop:GforgivenGi}. The formulas and auxiliary functions are partially remeniscent of computations of isoresidual fibers of moduli spaces of genus zero differentials \cite{CGPT25}. However, the technical task is cumbersome as the degeneration viewpoint singles out a 'maximal vertex' that governs the coloring in the bipartite complement of static edges, while the DR-volumes lead to expressions without such a distinguished vertex. Several layers of inclusion-exclusion rearrangements allow for the conversion of the two expressions.
\par
This part is concluded with Theorem~\ref{thm:completedvolgraphs2}, the recusion for completed volumes on Masur-Veech side in the case $n=2$.
\par
\par
\medskip
\paragraph{\textbf{An example:  The stratum $\mu=(5,3)$ with $k=2$}} We conclude with an example of a stratum of quadratic differentials that exhibits both types of graphs. 
\bas
D_1 & \=\left[\begin{tikzpicture}[
  baseline={([yshift=-.5ex]current bounding box.center)},
  scale=2,
  very thick,
  bend angle=30,
  comp/.style={circle,draw,thin,inner sep=2pt,font=\tiny},
  order bottom left/.style={pos=.05,left,font=\tiny},
  order top left/.style={pos=.9,left,font=\tiny},
  order bottom right/.style={pos=.05,right,font=\tiny},
  order top right/.style={pos=.9,right,font=\tiny},
  order node dis/.style={text width=.75cm},
]

\node[comp] (A) at (0,0) {3};

\node[comp,fill] (E) at (0,-0.75) {};


\node[minimum width=18pt,below left] (E-half) at (E.south west) {$5$};
\path (E) edge [shorten >=4pt] (E-half.center);

\node[minimum width=18pt,below right] (E-half) at (E.south east) {$3$};
\path (E) edge [shorten >=4pt] (E-half.center);

\draw (A) edge node[order bottom left,yshift=-1pt] {$8$} node[order top left] {$-12$} (E);
\end{tikzpicture}\right]
\quad
D_2 \=\left[\begin{tikzpicture}[
  baseline={([yshift=-.5ex]current bounding box.center)},
  scale=2,
  very thick,
  bend angle=30,
  comp/.style={circle,draw,thin,inner sep=2pt,font=\tiny},
  order bottom left/.style={pos=.05,left,font=\tiny},
  order top left/.style={pos=.9,left,font=\tiny},
  order bottom right/.style={pos=.05,right,font=\tiny},
  order top right/.style={pos=.9,right,font=\tiny},
  order node dis/.style={text width=.75cm},
]

\node[comp] (A) at (0,0) {2};
\node[comp] (B) at (0.5,0) {1};

\node[comp,fill] (E) at (0.25,-0.75) {};


\node[minimum width=18pt,below left] (E-half) at (E.south west) {$5$};
\path (E) edge [shorten >=4pt] (E-half.center);

\node[minimum width=18pt,below right] (E-half) at (E.south east) {$3$};
\path (E) edge [shorten >=4pt] (E-half.center);

\draw (A) edge node[order bottom left,yshift=-1pt] {$4$} node[order top left] {$-8$} (E);
\draw (B) edge node[order bottom right,yshift=-1pt] {$0$} node[order top right] {$-4$} (E);

\end{tikzpicture}\right]
\quad
D_3\ =\left[\begin{tikzpicture}[
  baseline={([yshift=-.5ex]current bounding box.center)},
  scale=2,
  very thick,
  bend angle=30,
  comp/.style={circle,draw,thin,inner sep=2pt,font=\tiny},
  order bottom left/.style={pos=.05,left,font=\tiny},
  order top left/.style={pos=.9,left,font=\tiny},
  order bottom right/.style={pos=.05,right,font=\tiny},
  order top right/.style={pos=.9,right,font=\tiny},
  order node dis/.style={text width=.75cm},
]

\node[comp] (A) at (0,0) {1};
\node[comp] (B) at (0.75,0) {1};
\node[comp] (C) at (1.5,0) {1};

\node[comp,fill] (E) at (0.75,-0.75) {};


\node[minimum width=18pt,below left] (E-half) at (E.south west) {$5$};
\path (E) edge [shorten >=4pt] (E-half.center);

\node[minimum width=18pt,below right] (E-half) at (E.south east) {$3$};
\path (E) edge [shorten >=4pt] (E-half.center);

\draw (A) edge node[order bottom left,yshift=-3pt] {$0$} node[order top left] {$-4$} (E);
\draw (B) edge node[order bottom right,yshift=-3pt] {$0$} node[order top right,yshift=12pt] {$-4$} (E);
\draw (C) edge node[order bottom right,yshift=-3pt] {$0$} node[order top right] {$-4$} (E);

\end{tikzpicture}\right] \\
D_4 &\=\left[\begin{tikzpicture}[
  baseline={([yshift=-.5ex]current bounding box.center)},
  scale=2,
  very thick,
  bend angle=30,
  comp/.style={circle,draw,thin,inner sep=2pt,font=\tiny},
  order bottom left/.style={pos=.05,left,font=\tiny},
  order top left/.style={pos=.9,left,font=\tiny},
  order bottom right/.style={pos=.05,right,font=\tiny},
  order top right/.style={pos=.9,right,font=\tiny},
  order node dis/.style={text width=.75cm},
]

\node[comp] (A) at (0,0) {2};
\node[comp] (B) at (0.5,0) {1};

\node[comp,fill] (E) at (0.25,-0.75) {};


\node[minimum width=18pt,above left] (A-half) at (A.north west) {$5$};
\path (A) edge [shorten >=4pt] (A-half.center);

\node[minimum width=18pt,below left] (E-half) at (E.south east) {$3$};
\path (E) edge [shorten >=4pt] (E-half.center);

\draw (A) edge node[order bottom left,yshift=-1pt] {$-1$} node[order top left] {$-3$} (E);
\draw (B) edge node[order bottom right,yshift=-1pt] {$0$} node[order top right] {$-4$} (E);

\end{tikzpicture}\right]
\quad
D_5\=\left[\begin{tikzpicture}[
  baseline={([yshift=-.5ex]current bounding box.center)},
  scale=2,
  very thick,
  bend angle=30,
  comp/.style={circle,draw,thin,inner sep=2pt,font=\tiny},
  order bottom left/.style={pos=.05,left,font=\tiny},
  order top left/.style={pos=.9,left,font=\tiny},
  order bottom right/.style={pos=.05,right,font=\tiny},
  order top right/.style={pos=.9,right,font=\tiny},
  order node dis/.style={text width=.75cm},
]

\node[comp] (A) at (0,0) {2};
\node[comp] (B) at (0.5,0) {1};

\node[comp,fill] (E) at (0.25,-0.75) {};


\node[minimum width=18pt,above left] (A-half) at (A.north west) {$3$};
\path (A) edge [shorten >=4pt] (A-half.center);

\node[minimum width=18pt,below left] (E-half) at (E.south east) {$5$};
\path (E) edge [shorten >=4pt] (E-half.center);

\draw (A) edge node[order bottom left,yshift=-1pt] {$1$} node[order top left] {$-5$} (E);
\draw (B) edge node[order bottom right,yshift=-1pt] {$0$} node[order top right] {$-4$} (E);

\end{tikzpicture}\right]
\quad
D_6=\left[\begin{tikzpicture}[
  baseline={([yshift=-.5ex]current bounding box.center)},
  scale=2,
  very thick,
  bend angle=30,
  comp/.style={circle,draw,thin,inner sep=2pt,font=\tiny},
  order bottom left/.style={pos=.05,left,font=\tiny},
  order top left/.style={pos=.9,left,font=\tiny},
  order bottom right/.style={pos=.05,right,font=\tiny},
  order top right/.style={pos=.9,right,font=\tiny},
  order node dis/.style={text width=.75cm},
]

\node[comp] (A) at (0,0) {1};
\node[comp] (B) at (0.5,0) {1};
\node[comp] (C) at (1.25,0) {1};

\node[comp,fill] (E) at (0.25,-0.75) {};
\node[comp,fill] (F) at (1,-0.75) {};


\node[minimum width=18pt,below left] (E-half) at (E.south east) {$5$};
\path (E) edge [shorten >=4pt] (E-half.center);

\node[minimum width=18pt,below left] (F-half) at (F.south west) {$3$};
\path (F) edge [shorten >=4pt] (F-half.center);

\draw (A) edge node[order bottom left,yshift=-2pt] {$1$} node[order top left] {$-5$} (E);
\draw (B) edge node[order bottom right,yshift=-2pt] {$0$} node[order top right] {$-4$} (E);
\draw (C) edge node[order bottom right,yshift=-2pt] {$0$} node[order top right] {$-4$} (F);
\draw (A) edge node[order bottom right, yshift=2pt] {$-1$} node[order top left,yshift=-1pt] {$-3$} (F);

  \end{tikzpicture}\right]
\eas
The first row shows the special star-graphs~$D_1,D_2,D_3$ while the second row shows the sunflower graphs~$D_4,D_5,D_6$. However, since the strata of type $\mu = (3,1)$ and type $(1,-1)$ in $k=2$ are empty, their volume contribution in the following table is zero. The first row in the table corresponding to the trivial graph~$\Gamma = \bullet$ is the contribution of the main component, the compactified moduli space $\ol{\cQ}(\mu)$.
\begin{center}
  \begin{tabular}{|c|c|c|c|c|}
  \hline
  \rule{0pt}{12pt} $\Gamma$ & $\frac{1}{|\Aut(\Gamma)|k^{h_\ab}}$ & $\prod_{e\in E(\Gamma)}\kappa_{e}$ & $\prod_{v\in V(\Gamma)}\text{vol}(\mu_v)$ & $\text{vol}({\Gamma})$ \\ \hline
  $\bullet$ & 1 & 1 & -35/648 & -35/648 \\ \hline
  $D_1$ & 1/2 & 10 & -305/18144 & -1525/18144\\ \hline
  $D_2$ & 1/4 & 12 & -1/160 & -3/160 \\ \hline
  $D_3$ & 1/48 & 8 & -5/288 & -5/1728  \\ \hline
  $D_4$ & 1/2 & 2 & -7/2160 & -7/2160 \\ \hline
  $D_5$ & 1/2 & 6 & 0 & 0  \\ \hline
  $D_6$ & 1/4 & 12 & 0 & 0 \\ \hline
\end{tabular}  
\end{center}
The total contribution $\vol(\Gamma)$ of each summand is the same for the component of the DR-cycle and for the Ribbon graph overcounting. Adding up we find
\bes
\ol{\vol}^\DR(5,3) \=  \frac{-73}{448}  \qquad \text{and} \quad \ol{\Vol}(5,3)
\= \frac{73}{420}\pi^{4}
\ees
in the (more standard) Masur-Veech volume normalization, see the conversion~\eqref{eq:Volvol}.
\par
\medskip
\paragraph{\textbf{Acknowledgments}} This project started at a conference in Toulouse (July~2025) for J.~Hubbard's 80th birthday, when the authors compared completed DR-cycle volume with values and formulas Elise Goujard and Ivan Yakovlev had obtained in their project \cite{DGY}. The authors thank them for sharing formulas and insights, and in particular Ivan for explaining their counting and graph coloring strategies in Section~\ref{sec:ribbon}. The authors also thank David Holmes for insights on the embedding of $\GMS$ into the space of rubber differentials. In particular, part of the deformation theory in Section~\ref{sec:defo} will also appear in a joint project of the first-named author with him and Samuel Grushevsky.

%% file: sec_background.tex
\section{Multi-scale differentials and the DR-locus} \label{sec:background}

We fix a \emph{type} $\mu = (m_1,\ldots,m_n)$ with $\sum m_i = k(2g-2)$ recording the tuples of orders of vanishing of the differential. Sometimes we switch to log convention $a_i = m_i+k$ and then use $A = (a_1,\ldots,a_n)$ to denote the \emph{log type.} Fixing the type also fixes the line bundle
\be \label{eq:defcL}
\cL \,:=\, \cL_\mu \= \omega_{X/\cM_{g,n}}^{\otimes k}\Big(-\sum_{i=1}^n m_i z_i\Big)\,.
\ee

\subsection{The space of multi-scale differentials} \label{sec:MS}

We let $\omoduli[g,n](\mu)$ be the stratum of differentials of type~$\mu$ with all of the~$n$~points labeled. We write $\bP\omoduli[g,n](\mu)$ for the projectivized stratum. More generally, for any $k \geq 0$ we let $\komoduli[g,n](\mu)$ be the stratum of $k$-differentials of the corresponding type. For $k=2$ we simply write $\cQ(\mu)$ for the stratum and $\bP\cQ(\mu)$ for its projectivization. (We drop the second index and just write $\omoduli(\mu)$ if we consider the unlabeled case.) These strata are sometimes disconnected. We will not need their component classification except for the following.
\par
If $\mu \in k\cdot \bZ_{\geq 0}^n$ we say that $\mu$ is \emph{of holomorphic abelian type}. If~$\mu$ is not of holomorphic abelian type we set 
\be
N\=2g-2 +n \qquad d = 2g-3+n  = N-1\,,
\ee
which will be the unprojectivized resp.\ projectivized \emph{dimension} of all the components of the stratum of differentials of type~$\mu$. If $\mu$ is of holomorphic abelian type, then there are \emph{abelian components} $\komoduli[g,n](\mu)_{\mathrm{ab}}$ of dimension
\be
N_\ab\=2g-1 +n  \qquad d_\ab \= 2g-2+n  \= N_\ab -1\,,
\ee
consisting of differentials which are $k$-th powers of abelian differentials. If $k>1$ there are for every $d|k$ the components consisting of primitive $k/d$-differentials. They have unprojectivized dimension $N = N_\ab-1$. Nevertheless we sometimes just write~$N$ or~$d$ to treat the dimension uniformly, if the type of the component is clear from the context.
\par
We next recall the essentials of definition of the space $\MS(\mu)$ of \emph{projectivized multi-scale differentials} as in \cite{LMS}\footnote{This space is called $\bP \Xi \omoduli[g,n](\mu)$ in loc.\ cit.\ The simpler symbol used here is used there for a relative coarse moduli space, locally the quotient by the twist group}. We start with the case $k=1$. We will use the notion of \emph{level graphs (of type~$\mu$)} with notation as in \cite{CMZEuler, To2} to describe boundary strata of spaces of differentials and subsequently also for objects in $\cat{Rub}_\cL$ over a nuclear base. In particular we denote the \emph{enhancements} (also known as \emph{prongs} or \emph{twists}) associated with the edge~$e$ by~$\kappa_e$. We index the levels by non-positive integers, starting with top level zero (hence they are $0,-1,\ldots,-L(\Gamma)$, but level passages and all parameters associated with them are indexed by positive integers. We define for each level passage\footnote{This notation differs from \cite{To2,LMS}, where the lcm's were called $a_i$, but it is consistent with e.g.\ \cite{CMZEuler}. It avoids clashes with the usual notation for the log type.}
\be
\ell_j \= \lcm(\kappa_e : \text{$e \in E(\Gamma)$ crossing the $j$-th level passage})
\quad \text{and} \quad \ell_\Gamma = \prod_{j=1}^L \ell_j\,.
\ee
A \emph{twisted differential compatible with the level graph~$\Gamma$} is a stable pointed curve $(X,\bfz)$ with level graph structure~$\Gamma$ on its dual graph together with collection of differentials $\bfomega = (\omega_{j})$ for $j=0,1,\ldots,L(\Gamma)$ with orders of zeros and poles at the marked points and prescribed by the $m_i$'s and $\kappa_e$'s at the edges, see \cite{LMS, CMZEuler}, satisfying the matching residue condition at horizontal edges of~$\Gamma$. A \emph{multi-scale differential} is a twisted differential compatible with some level graph (of type~$\mu$) that moreover satisfies  the \emph{global residue condition} 
\bes 
(\GRC_j): \qquad \sum_{e \vdash C} \Res_{q_e^-}(\omega_{-j}) = 0 \quad \text{for every component~$C$ above the $j$-th passage}
\ees
and a \emph{global prong-matching}, as defined below. We refer to loc.~cit.\ for the definition of morphisms and details for defining families of such objects, which we will not need in the sequel.
\par
\medskip
\paragraph{\textbf{Prong-matchings and framings}}
Consider a vertical edge~$e$ of~$\Gamma$. A \emph{local prong-matching~$\sigma_e$} is defined geometrically as an order-reversing bijections between the incoming horizontal prongs (aka separatrices) at the upper end~$q_+$ of the edge with respect to the differential~$\omega_+$ at the level of the upper end and the outgoing horizontal prongs at the lower end~$q_-$ of the edge with respect to the differential~$\omega_-$ at the level of the lower end. A \emph{global prong-matching} is a collection of prong-matchings for each vertical edge.
\par
Equivalently (see \cite[Section~5]{LMS}), viewing the prongs~$v_\pm$ as tangent vectors at~$q_\pm$, we can define a local prong-matching as a section
\be
\sigma_e \in \cN_e := q_+^* \omega_{X^+} \otimes q_-^* \omega_{X^-}
\quad \text{such that} \quad \sigma_e(v_+ \otimes v_-)^{\kappa_e} = 1\,,
\ee
where $X^\pm$ are the components of~$X$ at the two ends of the edge. To get rid of the (auxiliary) dependence of a horizontal direction, consider as in \cite[Section~3]{To2} the level differential at the upper and lower end as sections
$\tau_+ \in T^{-\otimes \kappa_e}_{q+}$ and $\tau_- \in T^{-\otimes \kappa_e}_{q-}$. Then
$\tau_e := \tau_+^{-1} \otimes \tau_- \in \cN_e^{\kappa_e}$. We can now state equivalently that a local prong matching is a section
\be
\sigma_e \in \cN_e \quad \text{such that} \quad \sigma_e^{\kappa_e}(\tau_e) = 1\,.
\ee
\par
A framing is an auxiliary finite datum that we add to a differential in order to separate the two ingredients that make up a prong-matching. This serves the goal of decomposing components of $\GMS(\Gamma)$ into levelwise products in Section~\ref{sec:support}.\footnote{Terminology is consistent with flat surface terminology e.g.\ in \cite{LMS, boissymero, BSW}, where the notion of \emph{framed} differential refers to a differential $(X,\omega)$ together with a choice of a (horizontal) prong (i.e.\ outgoing trajectory) at all (or a specified subset of the) zeros of~$\omega$. The same word 'framing' is used e.g.\ in literature related to stability conditions (e.g. \cite{BS15, BMQS}) for labelings of the real oriented blowup boundary of meromorphic quadratic differentials.}
\par
Consider a marked point~$z_i$ of order~$m_i$ of the abelian differential~$(X,\omega)$ and view~$\omega$ as a section of $\tau \in T_{z}^{-|m_i|-1}$. A \emph{local framing} at~$z_i$ is a section $\sigma \in  z_i^* T_X^\vee$ such that
$\sigma^{|m_i|+1}(\tau) = 1$. A \emph{global framing} is a local framing for a collection of marked points that is described by the context. For a twisted differentials this set of marked points will usually be the set of nodes at vertical edges.
\par
Suppose~$\Gamma$ is a level graph and $(X,\bfz,\bfomega)$ is a twisted differential compatible with~$\Gamma$. Then, given a global framing for each of the levels of~$\Gamma$ (at the marked points corresponding to the vertical edges) we obtain a natural global prong-matching by combining $\sigma_e = \sigma^+ \otimes \sigma_-$ the local framings at the two ends of an edge into a local prong-matching.
\par
\medskip
\paragraph{\textbf{$k$-differentials}}
We now recall properties of the compactification of the stata of $k$-differentials by multi-scale differentials for $k>1$, see \cite{CMSeclinsub} containing more details than \cite{kdiff}. Again we denote the compactification by $\MS_k(\mu) \supset \bP \komoduli[g,n](\mu)$ or just by $\MS_k$. It is the union of components of $k/d$-th powers of spaces of $d$-differentials. By induction it suffices to describe the primitive component $\MS_k^{\pr}(\mu)$, which we now present. 
\par
We let $\wh\mu$ in genus $\wh{g}$ with $\wh{n}$ marked points the type of the abelian differentials that arise as canonical $k$-fold covers of differentials of type~$\mu$. \cite[Section~7]{CMSeclinsub} defines a DM-stack $\widehat{\MS}(\wh{\mu})$ that maps to $\MS({\wh{\mu}})$ such that the image agrees with the closure of the locus of canonical coverings. These are the differentials~$\omega$ with the property that $\tau^* \omega = \eta_k \omega$ for some automorphism $\tau \in \wh{X}$ where $\eta_k$ is a primitive $k$-th root of unity.
\par
The marking of the $n$ points only partially determines the marking the~$\wh{n}$ points, allowing for a permutation group~$G$ exchanging in each fiber over the $n$ marked points cyclically the markings according to the $\bZ/k$-covering group. The quotient $\widehat{\MS}(\wh{\mu})/G$ is not yet the desired compactification of $\MS_k(\mu)$, since $\bP \komoduli[g,n](\mu)$ comes with a generic isotropy group of order~$k$ that are not present in the quotient. Adding these morphism to $\widehat{\MS}(\widehat{\mu})/G$ gives the DM-stack we want, which comes with a map
\be
s\colon \widehat{\MS}(\wh{\mu})\rightarrow \MS_k(\mu) \qquad \text{of degree
$\deg(s) = \frac{1}{k} \prod_{i=1}^n \gcd(m_i,k)$},
\ee
where the degree is the size of~$G$ over the degree of the non-representable map changing the generic isotropy group. As consequence of this construction we find:
\par
\begin{theorem} \label{thm:MSkproperties}
  The space of multi-scale $k$-differentials $\MS_k(\mu)$ is a smooth DM-stack. Its boundary $\MS_k(\mu) \setminus \bP \komoduli[g,n](\mu)$ is stratified by strata indexed by $\bZ/k$-coverings of level graphs $s_\Gamma: \wh{\Gamma} \to \Gamma$ as in \cite[Section~7.3]{CMSeclinsub}. Each such boundary stratum parametrizes abelian multi-scale differentials of type $\wh{\mu}$ compatible with $\wh{\Gamma}$ and with $\bZ/k$.
\end{theorem}
\par
From this theorem we need to retain that a multi-scale differentials of type $\wh{\mu}$ is the same as a twisted $k$-differential $(X,\bfz,(\omega_i)_{i=0}^L)$ of type~$\mu$ compatible with~${\Gamma}$ (together with prong-matchings) such that the lift to the canonical stable curve cover with dual graph~$\wh\Gamma$ satisfies the conditions $(\GRC_i)$ as above. See \cite{kdiff} for reformulations of the GRC in terms of the curve~$X$.

\subsection{Generalized multi-scale differentials, or the DR-locus} \label{sec:genMS}

We next aim for the definition of the space $\GMSmu$\footnote{Notation almost taken from \cite{To2}, where this symbol is used for the unprojectivized relative coarse moduli space. The space we work with is called $\bP\cat{Rub}_\mu$ in \cite{To2}, but it is hard to associated with that symbol a meaningful name. In \cite{ChiodoHolmes} this space is called the divisor relation locus $\DRL_\mu$, which we avoid to reserve DR-terminology to classes in Chow rings.}  of generalized multi-scale differentials. For this we work in the categrory $\cat{LogSch}$ of logarithmic schemes (see \cite{To2, ChiodoHolmes}) from which we freely borrow basic terminology and notation.
\par

We let $\cat{Rub}_0$ be the stack of rubber curves with outgoing slopes at the marked points all equal to zero, whose objects we denote as
$(\pi \colon \cX \to B, \, \beta\colon \cX \to \bG_{m,B}^\trop)$. We denote by
$\Picabs$ the relative Picard variety over~$\barmoduli[g,n]$. The Abel-Jacobi map associated with $O(\beta)$ and the line bundle~$\cL$ induce two sections $\sigma_0,\sigma_\cL: \cat{Rub}_0 \to \Picabs$, whose images lands in $\Picabs^0$, the subfunctor of $\Picabs$ consisting of line bundles of (total) degree~$0$. We define
\be
\cat{Rub}_\cL  = \cat{Rub}_0 \times_{\Picabs; \sigma_0, \sigma_\cL} \cat{Rub}_0
\qquad \text{and finally} \qquad \GMS_\mu \:=\ \bP(\cat{Rub}_\cL),
\ee
where the latter is the quotient stack by the natural $\bC^*$-rescaling action. A $B$-valued point in $\GMS$ is a tuple
\be \label{eq:Rubpoints}
(X/B,\beta,\cF,\varphi) \quad \text{with} \quad \pi^* \cF \cong \cL(-\beta)\,.
\ee
We summarize the main result of \cite{To2} and extend it to the case of $k$-differentials.
\par
\begin{theorem} \label{thm:MSinGMS}
For $k=1$ the space of multi-scale differentials is the 
 full subcategory of $\GMSmu$ given on objects by
\bes 
\MS(\mu)(B) \= \{(X/B,\beta,\cF,\varphi)\,:\, X'/B' \,\,\text{satisfies ($\GRC_i$) for $i=1,\ldots,L(\Gamma)$}\}
\ees
where $X'/B'$ is the pullback via any base change $B'\to B$ such that the family is nuclear.
\par
For $k>1$ the space of multi-scale differentials $\MS_k(\mu)$ is the normalization of the union of components of $\GMSmu$ that meet the locus of smooth curves. The image of $\MS_k(\mu)$ is the substack of $(X/B,\beta,\cF,\varphi)$ such that for any nuclear base changed family there exists some covering of level graphs $s_\Gamma : \wh{\Gamma} \to \Gamma$ such that the conditions $\GRC_i$ hold for the abelian differentials on the $k$-cover corresponding to~$\wh{\Gamma}$.
\end{theorem}
\par
Since $\cat{Rub}_0$, $\cat{Rub}_\cL$ and hence $\GMS_\mu$ are stacks over $\cat{LogSch}$, it would be more precise to state that $\MS(\mu)$ is a substack of the underlying algebraic stack of $\GMS_\mu$. See \cite[Section~4]{To2} for the necessary recap on minimal log structures. In order to even make sense of applying the condition $(\GRC_i)$ to an object as in~\eqref{eq:Rubpoints} we need the construction from the following proof.
\par
\begin{proof}[Sketch of the proof of Theorem~\ref{thm:MSinGMS}] Suppose $k=1$ and suppose that $X/B$ is nuclear and with a minimal log structure. In particular the map~$\beta$ specifies a level graph~$\Gamma$. Moreover the ghost sheaf
$\ghost_{B,b} = \bN \langle p_1,\ldots, p_L \rangle \oplus \bN \langle E^h(\Gamma)\rangle $ is the free monoid with one generator $p_i$ for each level passage and one for each horizontal edge of~$\Gamma$. To build the twisted differential one chooses a log splitting, i.e., a map $\wt{\psi} : \bN \langle p_1,\ldots, p_L\rangle \to \ghost$, such that the composition with the projection $\M_B \to \ghost$ is the natural inclusion. One checks that $\wt{\psi} (\sum_{i=1}^j \ell_i p_i)$ is a section of $\cO(\beta(v_j))$ and thus can be interpreted as a section of $\cO(\beta)$ on the complement $U_i$ of the top~$j$ levels of the stable curve~$X$. With this identification the collection of elements
\be
\omega_{-j} \= \varphi^{-1} \pi^* \wt{\psi} \Bigl(\sum_{i=1}^j \ell_i p_i\Bigr)
\ee
is the desired collection of twisted differentials. A different choice of log splitting results in a rescaling of these differentials (and the prong-matchings we) by an element of the level rotation torus and thus represents the same multi-scale differential.
\par
For the converse recall from \cite{LMS, To2} more precisely that a multi-scale differential also comes with a rescaling ensenble, a collection of \emph{edge parameters~$f_e$} for each vertical edge and \emph{level parameters}~$t_j$ for $j=1,\ldots, L(\Gamma)$. The map $p_j \mapsto t_j \in \cO(B)$  together with the compatibilities among~$f_e$ and~$t_i$ can be used to define a log structure with $\ghost_{B,b}$ as above on $X/B$ and we set $\beta(v_j) = \sum_{i=1}^j \ell_i p_i$. Defining~$\varphi$ is now a matter of unwinding definitions (see \cite[Section~5.6]{To2}).
\par
The upshot so far is that for $k=1$ we obtain an isomorphism of multi-scale differentials without the GRC-conditions with $\GMS$ and imposing the GRC's on both sides yields the claim. The sketch of the maps in both directions also illustrates that the difference between $1$-differentials and $k$-differentials is nowhere essential besides the definition of the $k$-differential version of the GRC, which in turn depends on the covering $s_\Gamma : \wh{\Gamma} \to \Gamma$, not only on~$\Gamma$. After having made such a choice, the correspondence still works and yields the claim.
\end{proof}
\par

\subsection{The tautological class and its variants} \label{sec:tautclass}

For $k=1$ the projectivized multiscale compactification~$\MS(\mu)$ maps to the projectivzed Hodge bundle
\be \label{eq:maptoHodge}
\rho: \MS(\mu) \to \bP(\cH_+), \qquad \cH_+ = \pi_*\omega_{X/\barmoduli[g,n]}\bigl(- \sum_{i:m_i <0} m_i z_i\bigr) 
\ee
which is obvious from the viewpoint of \cite{LMS} using the projectivization of the top level differential~$\omega_0$. We define the class $\xi$ to be the pullback of the first Chern class $\xi = \rho^* c_1(\cO_{\bP(\cH_+)}(-1))$ to~$\MS(\mu)$.\footnote{This is consistent with the use of this symbol in \cite{CMZEuler}, hence also in the sage-package \texttt{diffstrata}. Note however that the Masur-Veech volume papers \cite{CMSZ,CMSprincipal} use $\xi = c_1(\cO(+1))$.}
\par
For $k>1$ there are two tautological bundles in the literature. Replacing in~\eqref{eq:maptoHodge} the bundle $\cH_+$ by $\cH_{k,+} = \pi_*\omega^{\otimes k} (- \sum_{i:m_i <0} m_i z_i)$ gives a map $\rho_k: \MS_k(\mu) \to \bP(\cH_{k,+})$ and we define  $\zeta = \rho_k^* c_1(\cO_{\MS_k(\mu)}(1))$. Alternatively, the covering space $\wh{\MS}(\mu)$ comes with the restriction of the tautological class~$\xi$ from the ambient multi-scale space $\MS(\wh{\mu})$. The two bundles are related by
\be \label{eq:xizetaconv}
s^*\zeta \=-k\xi \quad \text{and} \quad s_* \xi \= -\frac{\deg(s)}{k} \zeta
\ee
see \cite[Equation~(50)]{CMSeclinsub}.
\par

In \cite{To2, ChiodoHolmes} the class $\eta = c_1(\cF^\vee) \in \CH^1(\GMS)$ is defined to be the dual first Chern class of the bundle~$\cF$ in the definition~\eqref{eq:Rubpoints}. In \cite{To2} was shown to be a pullback from a 'Hodge bundle' $\eta = \cO_{\bP(\cL(D))}(1)$, but for a twist $\cL(D)$ such that $R^1\pi_* \cL(D) = 0$. This vanishing is not true for the $\cH_+$ used here, but nevertheless the following compatibility holds and gives a justification for~\eqref{eq:maptoHodge} in the language of \cite{To2}.
\par
\begin{lemma} \label{le:etaISxi}
  The restriction of~$\eta$ of~$\MS_k(\mu)$ is equal to~$-\zeta$. In particular for $k=1$ we have $\eta|_{\MS} = -\zeta = \xi$.
\end{lemma}
\par
\begin{proof} Suppose $k=1$ and let $\cE = \cO_X \bigl(- \sum_{i:m_i <0} m_i z_i\bigr)$. Note that $\pi_* \omega_{X/\barmoduli[g,n]}(\cE)$ is indeed a vector bundle. (This is implicit in~\eqref{eq:maptoHodge}) above. A quick proof: By Serre duality it is enough to show that $R^1\pi_* O(-\cE)$ is locally free and commutes with base-change. By Cohomology and Base Change, it is enough to show that $\pi_* \cO(-\cE)$ is locally free and commutes with base change. But this is obvious, since~$\cE$ is an effective linear combination of sections.)
\par
Now, as in \cite{To2} we get natural maps
\be \label{eq:cFcH}
\pi^* \cF \to \cL(-\beta) \to \cL \to \cH_+
\ee
and by adjunction a map $\wt{\rho}: \cF = \pi_* \pi^* \cF \to \pi_* \cH_+$. To get the desired map  we check that $\wt{\rho}$ is universally injective (even though now the map in~\eqref{eq:cFcH} isn't). Since $\cL(-\beta)$ is a pullback from the base, every non-zero global section is non-vanishing on some top level component, one which the map to $\cH_+$ is indeed injective. This now gives the map~\eqref{eq:maptoHodge}. Now $\eta|_{\MS} = -\xi$ follows from the standard properties of maps to projective spaces \cite[\href{https://stacks.math.columbia.edu/tag/0FCY}{Example 0FCY}]{stacks-project}.
\par
For $k>1$ the same proof works with the bundle $\cH_{k,+}$ and the map $\rho_k$.
\end{proof}
\par

\subsection{The Hodge-DR formula}



We define the \emph{log DR-cycle} as the Gysin-pullback of along the zero section $\logDR_\cL = \sigma_0^! \sigma_\cL(\bP\cat{Rub}_\cL)$. The Hodge-DR-conjecture from \cite{To2} predicts a formula for products of the tautological classes against the log DR-cycle. The conjecture was proven in \cite[Theorem~1]{ChiodoHolmes}. It states more precisely the following. Let $\epsilon:  \barmoduli^{\cL,1/r} \to \barmoduli$ be the forgetful map from the space of $r$-the roots of the line bundle~$\cL$ and let $\cL^{1/r}$ be the universal bundle on the universal family (also denoted by~$\pi$) over the domain.
\par
\begin{theorem} \label{thm:ChiodoHolmes}
  The push-forward under the map $p: \cat{Rub}_0 \to \barmoduli[g,n]$ 
  \bes  \DR_{\cL,u} := p_* (\logDR_\cL \cdot \eta^u) \= [r^{u+1} \epsilon_* c_{g+u}(-R\pi_*\cL^{1/r})]_{r=0} \qquad \in \CH^{g+u}(\barmoduli[g,n])
  \ees
  where $[\cdot]_{r=0}$ denotes taking the constant term in $r$.
\end{theorem}
\par
The right hand side can be explicitly evaluated, as it is equal to a coefficient of a Chiodo-class (see \cite{Chiodo} and in this form \cite[Corollary~4]{JPPZ}), to which we refer for the details and terminology)
\ba \label{eq:ChClass}
\Ch_{g,n}(\mu) &\,:=\, \epsilon_* c(-R\pi_*\cL_\mu^{1/r}) \\
&\= \sum_{\Gamma} \frac{r^{2g-1+h(\Gamma)}}{|\Aut(\Gamma)|} \sum_{w \in W(\Gamma,r)}
c_{\Gamma,*} \prod_{e \in E(\Gamma)} \cC_e \prod_{v \in V(\Gamma)} \cC_v \,\,\prod_{i=1}^n \cC_i 
  \ea
  Here the first sum is over all stable graphs~$\Gamma$, the second is over all admissible weightings mod~$r$, denoted by~$W(\Gamma,r)$ and $c_\Gamma$ is the clutching map for~$\Gamma$. The edge, vertex and leg factors are respectively given by (if $e=(h,h')$ is the decomposition into two half-edges)
\ba
\cC_e &\= \Bigl(1-\exp\Bigl(\sum_{m \geq 1} (-1)^{m} \frac{B_m(w(h)/r)}{m(m+1)}(\psi_h^m - (-\psi_h)^m \Bigr)\Bigr)/(\psi_h + \psi_{h'})\\
\cC_v &\= \exp\Bigl(\sum_{m \geq 1} (-1)^{m} \frac{B_m(k/r)}{m(m+1)} \kappa_m \Bigr)\\
\cC_i &\= \exp\Bigl(\sum_{m \geq 1} (-1)^{m-1} \frac{B_m(a_i/r)}{m(m+1)} \psi_i^m \Bigr)
\ea
as elements in $\CH^*(\barmoduli[g,n])$. Here $B_n$ denotes the $n$-the Bernoulli polynomial. 
\par
We will be most interested in the case $u=2g-3+n$, which agress with the dimension of the stratum (except for the holomorphic abelian case), and in which case
the DR-cycle is a $0$-cycle. We thus abbreviate (including the holomorphic abelian case)
\be
\DR(\mu) \= [\DR_{\cL_\mu,2g-3+n}] \qquad \in \CH^0(\barmoduli[g,n])  \,.
\ee

%% file: sec_support.tex
\section{The support of $\GMS$} \label{sec:support}

In order to decompose the support $\GMS$ we define a collection of natural substacks indexed by level graphs~$\Gamma$. For certain~$\Gamma$, those which \emph{impose a GRC at every level passage} we will show in Proposition~\ref{prop:GMSunion} that the substacks~$\GMS(\Gamma)$ define a union of irreducible components of~$\Gamma$. As technical tool for the proofs and to split $\GMS(\Gamma)$ (up to commensurability) into a product of level-wise stacks we recast the notion of \emph{framing}, common in language of differentials, in the context of~$\GMS$.
\par
\medskip
\paragraph{\textbf{Frozen level passages}} Fix a level graph~$\Gamma$. Its level passages will be the \emph{frozen level passages} of the subfunctor $\GMS(\Gamma)$ we now define.
\par
Recall that for any level graph~$\Delta$ there is the operation of \emph{undegenerating} the $i$-th level passage, contracting all the edges crossing only this level passage and moving  (and merging, if connected by an edge) all vertices on the level right above and below this passage to one level. We denote by $\delta_i(\Delta)$ the result of undegenerating the $i$-th passage, a graph with $L(\Delta)-1$ levels below zero. The operation of undegeneration at different level passages commute and we denote by $\delta_J(\Delta)$ for any subset $J \subset \{1,\ldots,L(\Gamma)\}$ the result of undegenerating all level passages in~$J$. 
\par
Suppose that $\Delta$ is a degeneration of~$\Gamma$, more precisely that $\Gamma = \delta_J(\Delta)$. We the call the level passages in~$J$ the \emph{additional level passages} and the level passages in $\{1,\ldots,L(\Delta)\} \setminus J$ the \emph{frozen level passages} of~$\Delta$ (\emph{with respect to $\Gamma$}). 
\par
\begin{definition}
The stack of \emph{$\Gamma$-frozen multi-scale differentials} is defined to be the fibered category
\be \label{eq:GMSGamma}
\GMS(\Gamma) \subset \GMS
\ee
consisting of the objects $(X/B,\beta,\cF,\varphi) \in \GMS(B)$ such that for the pullback any nuclear family $X'/B'$ obtained via any base change has a level graph~$\Delta = \Delta(B')$ which is a degeneration of~$\Gamma$ and such that $(\GRC_i)$ holds whenever~$i$ is an additional level passage of~$\Delta$ with respect to $\Gamma$ and such that the level parameters satisfy $t_i^{\ell_i} = 0$ for every frozen level passage on any pullback of~$X/B$ to a nuclear family.
\end{definition}
\par
In this language $\MS = \GMS(\bullet)$ is the subfunctor where the trivial graph is frozen. We denote by $\iota_\Gamma : \GMS(\Gamma) \to \GMS$ the natural inclusion.
\par
\begin{prop} \label{prop:GMSsubstack}
For each level graph~$\Gamma$ the category of $\Gamma$-frozen multi-scale differentials $\GMS(\Gamma)$ is a proper algebraic substack of $\GMS$. Moreover $\GMS(\Gamma)$ is a union of irreducible substacks of~$\GMS$.
\end{prop}
\par
\medskip
\paragraph{\textbf{The reduced substacks}} We \emph{define} $\GMS(\Gamma)^\red$ to be locally the substack of $\GMS(\Gamma)$ cut of by~$t_i = 0$ for each level~$i$ that is a frozen level passage (or equivalently: the subspace where all edges in~$\Gamma$ persist as nodes). This symbol is justified by:
\par
\begin{prop} \label{prop:GMSred}
The substack $\GMS(\Gamma)^\red$ is smooth, in particular reduced.
\end{prop}
\par
Not all level graphs~$\Gamma$ can indeed serve as frozen levels, as the level passages might not be frozen, i.e.\ there exist deformations smoothing that level passage. This happens precisely, if the following conditions is met.
\par
\begin{definition} \label{def:GRCimpose}
If $k=1$ we say that $\Gamma$ \emph{weakly imposes a GRC at level~$j$} if the twisted differentials compatible with $\wh{\Gamma}$ that satisfy $(\GRC_j)$ form a subspace of positive codimension inside the space of twisted differentials. If $k>1$ we require the analogous condition on the cover: For any choice of $\bZ/k$-coverings of level graphs $s_\Gamma: \wh{\Gamma} \to \Gamma$ as in Theorem~\ref{thm:MSkproperties} the twisted differentials compatible with $\wh{\Gamma}$ that satisfy $(\GRC_j)$ form a subspace of positive codimension inside the space of twisted differentials which are $\tau$-eigendifferentials $\tau^* \omega = \eta_k \omega$.
\par
We say that~$\Gamma$ \emph{imposes a GRC at level~$j$}, if it does so weakly and if moreover there is an edge crossing precisely the $j$-th level passage, i.e, connecting level $-j$ to level $-j+1$.
\end{definition}
\par
A chain of four vertices on levels $0,-1,0,-2$ respectively with top levels holomorphic is an example of a graph that weakly imposes a GRC at all levels, but that does not (strictly) impose GRCs at all levels.
\par
In any case we call a level-graph \emph{full of GRCs} if it imposes a GRC at every level passage.
\par
\begin{prop} \label{prop:GMSunion}
The stack $\GMS$ is the scheme-theoretic union of $\GMS(\Gamma)$ with~$\Gamma$ running over the set of level graphs such that $\Gamma$ imposes a GRC at each of its level passages.
\end{prop}
\par
We remark that the number of level passages such that $\GMS(\Gamma)$ occurs in this decomposition is not bounded as~$g$ grows. To see this, take~$\Gamma$ starting with a chain of one vertex at each level below zero, with all markings on the very bottom, and each of the vertices connected to a pair of holomorphic abelian vertices on one level higher.
\par
For the proof we use the multi-scale viewpoint on~$\GMS$. An alternative proof should be possible using the deformation theory recalled in Section~\ref{sec:defo} and the deformation obstructions imposed by the residue conditions in \cite[Lemma~4.11]{HolmesSchmitt}. 
\par
\medskip
\paragraph{\textbf{Framings}} We start with the definition of the auxiliary space~$\GMS(\Gamma)_f$ of \emph{framed generalized multi-scale differentials with frozen level graph~$\Gamma$}. It is the finite cover of $\GMS(\Gamma)$ where all the edges of~$\Gamma$ are labeled and moreover at each level we choose a global framing~$\sigma^\pm_e$ as defined in Section~\ref{sec:MS} for the set of upper and lower ends vertical edges of~$\Gamma$ with the following property: Given $(X,\beta,\varphi) \in \GMS(\Gamma)$, then $(X,\bfz,\bfomega,(\sigma^\pm_e)_{e \in E(\Gamma)})$ is a preimage in~$\GMS(\Gamma)_f$ if there exists a representative $(X,\bfz,\bfomega,(\sigma_e)_e)$ of the prong-matched twisted differential (or equivalently: there exists a log-splitting) such that for each edge~$e$ the framings of the preimage point is related to the local prong matching at~$e$ by 
\be \label{eq:framingtoPM}
\sigma_e \= \sigma_\e^+ \otimes \sigma^-_e \,.
\ee
Since there are only finitely many pairs of framings for a given~$\sigma_e$ that satisfy~\eqref{eq:framingtoPM}, we deduce that forgetful map $c_\Gamma: \GMS(\Gamma)_f \to \GMS(\Gamma)$ (obviously being proper) is a finite map.
\par
We define the substack $\GMS(\Gamma)_f^\red \subset \GMS(\Gamma)_f$ by imposing $t_i=0$ just as in the unframed situation.
\par
\medskip
\paragraph{\textbf{Level-wise splittings}} The boundary divisors of the moduli space of multi-scale differentials $\MS$ are commensurable to a product, see \cite[Section~3]{CMZEuler}. Essentially the same construction works for the components of~$\GMS$. For a given~$\Gamma$ we let $\mu^{[i]}$ be the generalized type (in the sense of \cite[Section~5]{CMZEuler} of level~$-i$, i.e.\ parameterizing possibly disconnected curves, one for each vertex at level~$i$, together with differentials with zeros and poles of orders as prescribed by the half-edges (markings and edges) adjacent to those vertices.
\par
Note the difference to the description of boundary divisors in~$\MS$: For any set~$\frakR$ of residue conditions (as in \cite[Section~3]{CMZEuler}), i.e.\ subsets of the set of poles, such that for each subset the sum of residues is required to add up to zero) we let $\MS({\mu^{[i]}})^\frakR \subset \MS({\mu^{[i]}})$ be the subspace of differentials that obeys these residue conditions. These spaces appear in the commensurability diagram for boundary divisors of~$\MS$ (\cite[Proposition~4.4]{CMZEuler}), but no superscript~$\frakR$ appears on the bottom left of~\eqref{dia:k-comens}.
\par
\begin{prop} \label{prop:GMScomm}
  Let $\Gamma$ be a level graph that imposes a GRC at every level. Then $\GMS(\Gamma)$ is commensurable to the product of its level-wise strata, i.e. there is a diagram 
\be \label{dia:k-comens}
\begin{tikzcd}[column sep=small]
\GMS(\Gamma)^\red_f \ar[r] \ar[d,
"p_{\Gamma}"'] & \GMS(\Gamma)_f  \ar[d, "c_{\Gamma}"]  \\
  \prod_{i=0}^L \MS({\mu^{[i]}}) &
\GMS(\Gamma) \\
\end{tikzcd}
\ee
where $p_\Gamma$ and $c_\Gamma$ are finite maps. 
\par
The top horizontal map is locally the quotient map by the ideal generated by the level parameters~$t_i$, which are nilpotent of order $\ell_i$ in~$\GMS(\Gamma)_f$. Moreover, 
\be \label{eq:degpdegc}
\frac{\deg(p_\Gamma)}{\deg(c_\Gamma)} \= \frac{\prod_e \kappa_e}{|\Aut(\Gamma)| \ell_\Gamma}\,.
\ee
\end{prop}
\par
We conclude the description of~$\GMS(\Gamma)$ with an easy general dimension observation. A vertex $v \in V(\Gamma)$ is called \emph{holomorphic abelian} if all the half-edges adjacent to it have a label in~$\bZ_{>0}$ and divisible by~$k$. Generalizing this, we call a subset of vertices $v \subset V(\Gamma)$ which is connected through horizontal edges (only) \emph{holomorphic abelian (horizontally connected) component} if all the half-edges adjacent to besides the horizontal edges it have a label in~$\bZ_{>0}$ and divisible by~$k$. We write $V(\Gamma)^\ab$ for the set and write
\be \label{eq:defhab}
h_{\textrm{ab}} \= |V(\Gamma)^\ab|
\ee
for number of holomorphic abelian components in~$\Gamma$. They are obviously local maxima of~$\Gamma$. A component~$C$ of $\GMS(\Gamma)$ determines a subset
\be \label{eq:VabC}
V(\Gamma)^\ab_C  \, \subseteq \, V(\Gamma)^\ab
\ee
of top level vertices where the differential actually is a $k$-th power of an abelian differential.
\par
\begin{lemma} \label{le:dGammadef}
The components of $\GMS(\Gamma)$ have projectivized dimensions 
\be \label{eq:dGammadef}
d_\Gamma := \dim(\GMS(\Gamma)) \= d - L(\Gamma)  + |V(\Gamma)^\ab_C|\,.
\ee
in the range $0 \leq |V(\Gamma)^\ab_C| \leq h_{\textrm{ab}}$. 
The lower bound is always attained for some component if~$k>1$. The upper bound is attained if and only if each holomorphic abelian vertex has an abelian component.
\end{lemma}
\par
\medskip
\paragraph{\textbf{Proofs}} The proofs of the propositions will be combined as follows.
\par
\begin{proof}[Proof of Proposition~\ref{prop:GMSsubstack}, Proposition~\ref{prop:GMSred}, Proposition~\ref{prop:GMSunion}, and of Proposition~\ref{prop:GMScomm}] The properness in  Proposition~\ref{prop:GMSsubstack} is proven as in the case of $\MS = \GMS(\bullet)$ in \cite{LMS}: Under degeneration the nodes at the new level passages are forced to satisfy the GRC by Stokes' theorem. The additional condition $t_i^{\ell_i} = 0$ is a closed condition, too.
\par
The remaining statements of the first three propositions are local, to be verified at every point~$x \in \GMS$. Let~$\Gamma$ be the level graph at this point. The key idea is visible in:
\par
\medskip
\paragraph{\textbf{Case: the graph~$\Gamma$ has just one level passage}} Suppose the top level has, say,  $b$ components that impose a~$\GRC$. Take the first component~$C$ and consider the edges adjacent to this component. For each of those edges the normal form \cite[Equation~(3.11)]{To2}\footnote{A priori an element in $\GMS$ satisfies the normal formal \cite[Equation~(3.10)]{To2}. If $f_e^{\kappa_e} \neq 0$ it is shown in \cite[Theorem~4.3]{LMS} how to change coordinates to obtain the first normal form. If $f_e^{\kappa_e} =0$ the iterative Ansatz \cite[Remark~4.6]{LMS} needs to be run only for~$\ell$ steps and is thus easily shown to produce a convergent base change.}  of the node $u_e v_e = f_e$ and the differentials
\be \label{eq:omegaNF}
\omega_0 \= \bigl(u_e^{\kappa_e} + f_e^{\kappa_e} r_e\bigr)\frac{du_e}{u_e} \quad \text{and} \quad 
\omega_1 \= -\bigl(v_e^{-\kappa_e} + r_e\bigr)\frac{dv_e}{v_e}
\ee
implies that the period of the corresponding vanishing cycles equals $f_e^{\kappa_e} r_e= t_1^{\ell}r_e$ by the relation between edge parameters and level parameters. Stokes' theorem implies that the sum of those periods vanishes and we conclude that
\bes
t_1^{\ell} \cdot \sum_{e \vdash C} r_e \= 0\,.
\ees
We conclude that either~$\Gamma$ imposes a GRC, the nodes all persist, and $t_1^\ell = 0$, i.e, the defining conditions of~$\GMS(\Gamma)$ hold, or the sum of residues vanishes for each components, i.e., $(\GRC_1)$ holds. (If $b=0$, the first possibility does not occur.)

We need to show that the two components are irreducible locally near~$x$. The component where $(\GRC_1)$ holds is just~$\GMS(\bullet) = \MS$ and this follows from the smoothness of this component, the main result of~\cite{LMS}.
\par
To show the irreducibility of $\GMS(\Gamma)$ near~$x$ we will show locally an isomorphism to a neighborhood of a point in $\MS(\mu^{[0]}) \times \MS(\mu^{[1]}) \times \Spec \bC[t_1]/t_1^\ell$. Starting with an element of~$\GMS(\Gamma)$ we obtain the elements in $\MS(\mu^{[i]})$ by taking the differentials with the normal forms~\eqref{eq:omegaNF} on the nodes, replacing $u_ev_e = f_e$ by $u_e v_e=0$ everywhere. Thanks to $t_1^{\ell} = 0$ the differential~$\omega_0$ is actually holomorphic (i.e.\ without simple poles) at the nodes. Taking the normalization of the family gives the two pieces we need. The level parameter~$t_1$ is anyway part of the datum of the point in~$\GMS(\Gamma)$.
\par
Conversely, given a pair $(X_i,\omega_i)$ of multi-scale differentials in $\MS(\mu^{[0]}) \times \MS(\mu^{[1]})$, we glue the Riemann surfaces~$X_i$ at the appropriate zeros and poles of~$\omega_i$ to form a stable curve~$X'$ with dual graph~$\Gamma$. The family of degenerating curves~$X$ will be obtained from~$X$ by replacing the nodes $u_e v_e=0$ with the (infinitesimally) degenerating family $u_ev_e = f_e$ where $u_e$ and $v_e$ are coordinates that put the differentials~$\omega_0$ and~$\omega_1$ in normal form and where $f_e = t_1^{\ell/\kappa_e}$. With this form, the differential~$\omega_0$ glues to $0 = t_1^{\ell} \omega_1$ on the other side of the node.
\par
To show that the two constructions are inverses to each other, we use the essential uniqueness of the normal forms and that the constructions are local. More precisely, the ('Strebel') normal form depends on the upper (holomorphic) end on the choice of $\kappa_e$-th root of unity \cite[Theorem~4.1]{LMS}. On the lower (meromorphic) end, it depends on the choice of an auxiliary section. To pin these choices we fix once and for all in the special fiber over~$x$ (which is not altered in the construction, just as the whole locus over $t_1=0$) neighborhoods of the nodes and choose in annular boundary collars some reference points that we require to match with points in the 'plumbing fixtures' $u_ev_e = f_e$ in order kill non-uniqueness in the Strebel normal form. 
\par
To summarize: the reader will recognize that this is exactly the plumbing construction of \cite[Section~10.4]{LMS} without using any modification differentials but using $f_e^{\kappa_e}=0$ instead. This concludes the proof of irreducibility near~$x$. The smoothness in Proposition~\ref{prop:GMSred} also follows from this construction. 
\par
\medskip
\paragraph{\textbf{Case: general graph~$\Gamma$}} We start with a discussion of the  candidates for irreducible components near~$x$. For the first level passage, we conclude as above that $t_1^{\ell_1} = 0$ holds or $(\GRC_1)$ holds. Suppose $\Gamma$ has two level passages.
\par
If the second level passage imposes no GRC, we can smooth this level passage and parameterize a neighborhood by the union of $\MS$ (if $x$ is a point where $(\GRC_1)$ holds) and $\GMS(\delta_2(\Gamma))$. To show that this space is irreducible, we use the local identification with $\Spec \bC[t_1]/t_1^\ell \times \Spec \bC[t_2] \times \prod_{i=0}^2 \MS(\mu^{[i]})$. This identification is proven combining the proof in the case $L(\Gamma)=1$ above for smoothing the nodes in the first level passage with the usual parameterization of boundary neighborhoods by plumbing (without a nilpotent parameter, but with $t_2 \in \Delta$ instead) in the multi-scale space.
\par
If the second level passage strictly imposes a GRC, we compare $\omega_1$ and~$\omega_2$ as in~\eqref{eq:omegaNF} and now arrive at the conclusion that for each component~$C$ above the second level passage
\bes
t_2^{\ell_2} \cdot \sum_{e \vdash C} r_e \= 0\,
\ees
holds. Together we conclude that~$t_1^{\ell_1} = 0$ or $(\GRC_1)$ holds, and $t_2^{\ell_2} = 0$ or $(\GRC_2)$ holds. Equivalently, we see that all components that meet at~$x$ belong to one of the four spaces $\MS$, $\GMS(\delta_1(\Gamma))$, $\GMS(\delta_2(\Gamma))$, and $\GMS(\Gamma)$ (if both $(\GRC_1)$ and $(\GRC_2)$ hold at~$x$, and the obvious subsets, if not). For three of them we have already shown this in the previous case. For the last, we prove a local identification with $\Spec \bC[t_1]/t_1^{\ell_1} \times \Spec \bC[t_2]/t_2^{\ell_2} \times \prod_{i=0}^2 \MS(\mu^{[i]})$ iterating the plumbing construction above over two level passages.
\par
If the second level passage imposes a GRC, but only weakly so, the obstruction by Stokes theorem now applies only to edges comparing~$\omega_0$ and~$\omega_2$ and from which we deduce
\bes
t_1^{\ell_1} t_2^{\ell_2} \cdot \sum_{e \vdash C} r_e \= 0\,.
\ees
Together we conclude that one of 
\be \text{$(\GRC_1)$ and $(\GRC_2)$}, \qquad \text{$(\GRC_1)$ and $t_2^{\ell_2} = 0$}, \qquad
\text{or $t_1^{\ell_1} = 0$}
\ee
holds. We observe that (if both $(\GRC_1)$ and $(\GRC_2)$ hold at~$x$), all components that meet at~$x$ belong to one of the three spaces $\MS$, $\GMS(\delta_1(\Gamma))$ or $\GMS(\delta_2(\Gamma))$. Using the local irreducibility of these spaces (same argument as above again), this proves Proposition~\ref{prop:GMSunion} for this case and at this point, which claims that we do not need $\GMS(\Gamma)$ to describe~$\GMS$ as a union.
\par
We leave it to the reader to check that for $\Gamma$ with an arbitrary number of levels the same pattern continues: Strictly imposing a GRC at every level implies the $\GRC_i$ or $t_i^{\ell_i}=0$-dichotomy at every level and imposing  $t_i^{\ell_i}=0$ gives a component. In the absence of an edge joining level $-j$ with $-j+1$ one of the conditions involves a product of level parameters. This implies that corresponding~$\GMS(\Gamma)$ is still locally irreducible (Proposition~\ref{prop:GMSsubstack}) but contained in~$\GMS(\delta_j(\Gamma))$. 
\par
To prove the remaining claims of Proposition~\ref{prop:GMScomm} observe that on one the loci $t_i=0$, i.e.\ on~$\GMS(\Gamma)^\red_f$ we do not need to perform any local surgery at the edges crossing the frozen level passages. This implies that the map~$p_\Gamma$ can be defined globally. The fiber of~$p_\Gamma$ is the set of possible framings at the singularities that stem from nodes of~$\Gamma$ of the image element in the product of $\MS(\mu^{[i]})$. The degree computation follows as in \cite[Lemma~3.6]{CMSeclinsub}.
\end{proof}
\par
\medskip
\paragraph{\textbf{Chow rings and tautological bundles}} An immediate consequence of Proposition~\ref{prop:GMScomm} is:
\par
\begin{cor} The rational Chow ring of $\GMS(\Gamma)$ is a product
\be
\CH^*(\GMS(\Gamma)) \= \prod_{i=0}^L \CH^*(\MS(\mu^{[i]}))
\ee
of the level-wise Chow rings. 
\end{cor}
\par
With this identification we denote by $\zeta^{[i]} \in \CH^1(\GMS(\Gamma)) $ the pullback of the class $\zeta \in \CH^1(\MS(\mu^{[i]}))$ as defined in Section~\ref{sec:tautclass} to~$\GMS(\Gamma)$. We abbreviate $\zeta^\top = \zeta_\top = \zeta^{[0]}$ and $\zeta^\bot = \zeta_\bot = \zeta^{[L(\Gamma)]}$. The same as the proof of Lemma~\ref{le:etaISxi} shows:
\par
\begin{lemma} \label{le:etaISxitop}
  The restriction of~$\eta$ of~$\GMS$ is equal to~$-\zeta^\top$.
\end{lemma}
\par
We also need the tautological bundles at each of the vertices of~$\Gamma$, or more generally for any subset $S$ of the vertices at any level~$i$ of~$\Gamma$.
We define $\cO_{\mu|_S}(1)$ to be the bundle on $\GMS(\Gamma)$ whose fiber over $(X/B,\beta,\varphi,\cF)$ are the multiples of the tautological differential~$\omega_i$ restricted to the vertices in~$S$.
Then, using a bullet point to remind us that we have a bundle on the typically non-reduced $\GMS(\Gamma)$ rather than on the left column of~\eqref{dia:k-comens},  we let
\be \label{eq:cO1S}
\cO_{\mu|_S, \bullet}(1)  \= \cL(D) \quad \text{and} \quad \zeta_{S,\bullet} = c_1(\cO_{\mu|_S, \bullet}(1))
\ee
Checking on local generators (as in \cite[Proposition~4.9]{CMZEuler}) shows:
\par
\begin{cor} \label{cor:xiconversion}
In the case $S = V(\Gamma)^\top$ the set of top level vertices, the tautological classes on the various spaces agree up to a scalar, i.e.
\be
\zeta_{V(\Gamma)^\top,\bullet} \= \zeta^\top \= \frac{\prod_e \kappa_e}{|\Aut(\Gamma)| \prod_{i=1}^L \ell_i} \, \cdot \, c_{\Gamma,*} \iota_* p_\Gamma^* \zeta_{\MS^{[0]}} \quad 
\in \CH^1(\GMS(\Gamma))
\ee
\end{cor}

%% file: sec_normalsheaf.tex
\section{The normal sheaf to $\GMS(\Gamma)$} \label{sec:normal}

The goal of this section is to compute in Theorem~\ref{thm:nbGMSapprox} the normal sheaf to the subspaces $\GMS(\Gamma)$ to the extent necessary to evaluation in Section~\ref{sec:excess} the top intersection numbers with the $\zeta$-classes. We will recall in Section~\ref{sec:defo} some deformation theory needed for the proofs. The statements we aim for are in Section~\ref{sec:tangnormal}. The following diagram summarizes the spaces we will be working with:
\begin{center}
\begin{tikzcd}
&&& \bP \cat{Rub}(\Gamma) \ar{d} \\ 
\GMS(\Gamma)^\red  \ar[r,"\iota_\red"]  \ar[rrru, bend left, pos=0.2, "\frakj_\Gamma^{\red}"]&
\GMS(\Gamma)    \ar[r,"\iota_\Gamma"] \ar[rr, bend left,pos = 0.3, "\frakj_\Gamma"] &
\GMS  \arrow{r}{\frakj_{0}} 
\arrow{d}{\frakj_{\cL}}  & \bP \cat{Rub}  \arrow{d}{\sigma_\cL}\\
&& \bP \cat{Rub} \arrow{r}{\sigma_0} &  \Picabs
	\end{tikzcd}
\end{center}
\par
We let $\bP \cat{Rub}(\Gamma)$ be the closed substack of $\bP \cat{Rub}(\Gamma)$ where the underlying level graph is a degeneration of~$\Gamma$ and where the nodes corresponding to edges crossing a level passage of~$\Gamma$ are persistent, i.e., cut out by $t_i = 0$ for each level parameter~$t_i$. By definition the image of $\bP \cat{Rub}(\Gamma)$ in $\barmoduli[g,n]$ is contained in the boundary. We record the well-known resp. obvious facts:
\par
\begin{prop} \label{prop:GMSredsmooth}
  The stack $\bP \cat{Rub}$ is smooth and $\bP \cat{Rub}(\Gamma)$ is a smooth substack.
  \par
  The intersection of $\GMS(\Gamma)$ with $\bP \cat{Rub}(\Gamma)$ is the smooth substack  $\GMS(\Gamma)^\red$. 
\end{prop}
\par
\begin{proof} The first claim is essentially due to Marcus-Wise \cite{MarcusWiselog}, see \cite[Theorem~2.4]{To2} for the details to match with the definition given here.
  \par
Using the identification of $\GMS$ with multi-scale differentials (without GRC) from \cite{To2} is suffices to proof the smoothness of the subspace in that space where the level graph is a degeneration of~$\Gamma$ and the nodes of~$\Gamma$ persist. This follows as in \cite{LMS}, using perturbed period coordinates (or log period coordinates) on all levels of~$\Gamma$.
\end{proof}
\par

\subsection{Some deformation theory} \label{sec:defo}

We need some notation to describe various subspaces of the tangent space of $\bP\cat{Rub}$. The material here extends the local computations of Holmes-Schmitt \cite{HolmesSchmitt} to the log language. (Differences in the statements are due to the fact that \cite{HolmesSchmitt} work on a space between $\bP \cat{Rub}$ and $\barmoduli$, essentially the space $\bP \cat{Div}$ from \cite{MarcusWiselog}.)
\par
First, recall from Section~\ref{sec:support} the definition of the set of holomorphic subgraphs $V(\Gamma)^{\ab}$ in~$\Gamma$. These are a subset of components $\pi_0(\Gamma\setminus E^v)$ a connected component of the level graph $\Gamma \setminus E^v$ where all the vertical edges have been removed. For $\ul{v} \in \pi_0(\Gamma\setminus E^v)$ we let $X_{\ul v}$ be the corresponding connected component of the partial normalization
\begin{equation}\label{eq:widetildeX}
\nu:\ \widetilde{X}^\nu \coloneqq \bigsqcup_{\ul{v} \in \pi_0(\Gamma\setminus E^v)}  X_{\ul v}  \to X
\end{equation}
of~$X$ at all the vertical nodes, and let $\omega_{X_{\ul v}}$ be its canonical sheaf.  Then
\begin{equation*}
\omega_{X,\mathrm{hor}} \,\coloneqq \, \nu_* \bigoplus_{\ul{v} \in \pi_0(\Gamma\setminus E^v)} \omega_{X_{\ul v}} \subset \omega_X 
\end{equation*}
is the subsheaf with (order~$k$) poles along the horizontal nodes, and no poles at vertical nodes.
\par
Second, for any abelian group~$(A,\cdot)$ and any level graph~$\Gamma$ with level set $L(\Gamma)$ and edge set $E(\Gamma) = E^v(\Gamma) \sqcup E^h(\Gamma)$, split as usual into vertical and horizontal edges, there is a natural map defined on
level generators by
\begin{equation}\label{eq:rhoGamma}
\rho_\Gamma: A^{L(\Gamma)} \to A^{E^v(\Gamma)}, \qquad \lambda_0 \mapsto 1, \quad
\lambda_{-i}\mapsto \begin{cases}
\lambda_{-i}^{\kappa_e} & \text{if~$e$ crosses the $i$-th passage} \\
1 & \text{else} \\
\end{cases}
\end{equation}
where $\kappa_e$ are the enhancements of~$\Gamma$. We define
\begin{equation}H^1_L(\Gamma, A) \coloneqq \coker(\rho_\Gamma)\,.
\end{equation}
\par
From now one we fix again a level graph~$\Gamma$ specifying frozen levels, we suppose that~$\Gamma$ imposes a GRC at every level, and consider a degeneration~$\Pi$ of~$\Gamma= \delta_I(\Pi)$ where~$I$ are the additional level passages. 
\par
\begin{lemma} Suppose that $x = (X,\beta,\cF,\varphi) \in \GMS(\Gamma)$ has as level graph a degeneration~$\Pi$ of~$\Gamma = \delta_I(\Pi)$. Then the tangent space to $\bP \cat{Rub}$ splits as
\be
T_x \bP \cat{Rub} \= H^1_L(\Pi,\bC) \oplus H^1\Bigr(X, \omega_{X,\mathrm{hor}}^\vee(-\bfz)\Bigr) \oplus \bC^{L(\Delta)} \oplus \bC^{E^h(\Delta)}
\ee
Inside this space the tangent space to $\bP \cat{Rub}(\Gamma)$ is given by
\be
T_x \bP \cat{Rub}(\Gamma) \= H^1_L(\Pi,\bC) \oplus H^1\Bigr(X, \omega_{X,\mathrm{hor}}^\vee(-\bfz)\Bigr)  \oplus \bC^I\,.
\ee
\end{lemma}
\par
We denote by $\iota_G$, $\iota_\Omega$,  $\iota_L$  and $\iota_h$ (resp.\ $\iota_I$) the inclusion of these subspace into $T_x \bP \cat{Rub}$ (resp.\ into $T_x \bP \cat{Rub}(\Gamma)$). They correspond to changing the log structure in the graph, to equisingular deformations, and to level opening-up respectively. 
\par
\begin{proof} The deformation given by the first summand fixes the underlying stable curve~$X$ and the map~$\beta$, but changes the log structure. More precisely, recall that the log structure at a node~$q$ of~$X$ corresponding to the edge $e \in E^v(\Pi)$ needs an element $D_e \in \M_B$ with two properties. First, its reduction is the element $\delta_e \in \ghost$ and, second, image is $\cO_B$ is a smoothing parameter for the edges. With the help of these elements the log structure is then defined by 
\be
\M_{X,b}  \= \{ (u,v) \in \M_{B,b}^2 \,\,\text{such that there exists $n \in N$ with $D_e^n = u/v$} \}
\ee
For any tuple $\lambda \in (\bC^*)^{E(\Pi)}$ we define a deformation of the log structure by sending $D_e \mapsto \lambda_e D_e$. Two such deformations define isomorphic log structures if there is an automorphism of the underlying monoid that takes one into the other. Recall that $\ghost_{B,b} = \bN \langle p_1,\ldots, p_L \rangle \oplus \bN \langle E^h(\Gamma)\rangle$. The automorphism induced from $E^h(\Gamma)$ cannot be used to relate two deformations of the vertical edge parameters. For each of the~$p_i$ there is a copy of~$\bC^*$ in $\Aut(\M_{B,b})$ acting by $t\cdot(a \cdot p_i,u) \mapsto (a\cdot p_i, ut^a)$. Unwinding definitions this action gives precisely the image of $\rho_\Pi$ on the deformations of~$D_e$. Taking the tangent map to this deformation gives the injection~$\iota_G$.
(We remark that similar deformations of the horizontal edges are fully compensated by the automorphism induced from $E^h(\Gamma)$. For this reason they do not show up in $H^1_L(\Pi,\bC)$, only in the deformations of the node.)
\par
The second summand are the equisingular deformations (i.e.\ touching neither the nodes nor the log structure, but changing each piece of partial normalization~$\nu$ in~\eqref{eq:widetildeX}. Standard deformation theory shows that these are described by the inclusion~$\iota_\Omega$.
\par
Finally, the (non-canonical, depending on choices of local coordinates) inclusion $\iota_L$ and $\iota_h$ arise from opening the nodes (or lifting those for $\barmoduli)$. Obviously only the restriction to $\bC^I$ is tangent to $\bP\cat{Rub}(\Gamma)$.
\par
For the injectivity of the sum of those four maps observe that precisely the image of $\iota_G$ is the kernel of the tangent map of the forgetful map to $\barmoduli$. For the other three, the injectivity is obvious on $\barmoduli$. That the images of these maps span the tangent space follows from a dimension count.
\end{proof}
\par
Next we aim to characterize the tangent space to a component~$C$ of $\GMS(\Gamma)$ inside this space. For this purpose let
\be
b: \T_x \bP \cat{Rub} \to \T_x^\rel \Picabs \subseteq \T_x \Picabs
\ee
be the difference of the tangent maps to the sections $\sigma_0$ and~$\sigma_\cL$ at the point~$x$. Its image lies in the relative tangent space, the kernel of the tangent map $ \T \Picabs \to \T \bP \cat{Rub}$ to the forgetful map.
\par
Moreover we let $b_G \coloneqq b \circ \iota_G$ and $b_\Omega \coloneqq b \circ \iota_\Omega$ and decorate with a superscript~$\Gamma$ the restrictions to the tangent space of $\bP\cat{Rub}(\Gamma)$. To compute the image and cokernel of~$b$ we will work with the dual map and identify its kernel. A key observation is that the identification of the kernel with $\Omega^{\max }$ in the following lemma holds for every point in~$\GMS(\Gamma)$, whether the level graph is~$\Gamma$ or a degeneration of it.
\par
\begin{lemma} \label{le:bdualindentification}
 The kernel of $b_G^\vee$ and hence also the kernel of $(b_G^\Gamma)^\vee$ is equal to $H^0(X,\omega_{X,\mathrm{hor}})$, embedded into $H^0(X,\omega_X)$ via the map 
\begin{equation} \label{eq:kerbGammavee}
\bigoplus_{\ul{v} \in \pi_0(\Gamma\setminus E^v)} H^0(X_{\ul v},\omega_{X_{\ul v}}) \hookrightarrow H^0(X,\omega_X)\,,
\end{equation}
yet equivalently is the space of sections of $\omega_X$ whose residue vanishes at every vertical node.
\par
Moreover, the kernel of $b_G^\vee \oplus b_\Omega^\vee$ and hence also the kernel of $(b_G^\Gamma)^\vee \oplus (b_\Omega^\Gamma)^\vee$ is given by the image of 
\begin{equation} \label{eq:OmegaMaxincl}
\Omega_{\ab,C} \coloneqq \oplus_{\ul{v} \in V(\Gamma)_C^\ab}\, \bC \cdot \omega^{1/k}|_{X_{\ul{v}}}\hookrightarrow H^0(X,\omega_X)\,.
\end{equation}
\end{lemma}
\par
In the last statement observe that for $\ul{v} \in V(\Gamma)_C^\ab$ the $k$-th root of~$\omega$ exists, even globally on the stratum. 
\par
\begin{proof} The statements about $b_G$ and $b_\Omega$ are the content of \cite[Lemma~4.1, 4.6 and~4.9]{HolmesSchmitt}. Since the restriction to the tangent space $\T_x \bP \cat{Rub}(\Gamma)$ only restricts from $L_\Delta$ to $L_I$ but not the other summands, the same holds for $b_G^\Gamma$ and $b_\Omega^\Gamma$.
\end{proof}
\par
To match with earlier notation, we remark that $h^\ab = \dim_\bC \Omega^\ab$ if in the component~$C$ all holomorphic local maxima of~$\Delta$ with labels divisible by~$k$ carry indeed $k$-th powers of abelian differentials.  
\par
\begin{cor} \label{cor:tangentspacecontained}
The tangent space $T_x\GMS$ at $x = (X,\beta,\cF,\varphi)$ belonging to the component~$C$ of~$\GMS(\Gamma)$ is contained in the kernel of the map  
\be
b_{\ab,C}: T_x \bP \cat{Rub} \to \Omega^\vee_{\ab,C}\,.
\ee
Consequently, the tangent space to $T_x\bP \cat{Rub}(\Gamma)$ is also contained in this kernel.
\end{cor}
\par
\begin{proof}
This is a consequence of the dualization computation
\begin{equation}\label{eq:cokerb}
(\coker(b))^\vee \= \ker(b^\vee) \subseteq \ker(b_\Omega^\vee) \cap \ker(b_G^\vee) \= \ker(b_\Omega^\vee \oplus b_G^\vee)\,.
\end{equation}
and Lemma~\ref{le:bdualindentification}.
\end{proof}
\par

%
%
%
%
The following propositions complete the local picture of the components of $\GMS$ inside $\bP \cat{Rub}$. They are not necessary for volume computations, but relevant for determining the full excess intersection class.
\par
\begin{prop} \label{prop:GMSGammaregemb}
For a graph~$\Gamma$ which is full of GRC's the restriction $\frakj_\Gamma$ of the closed embedding $\frakj_0: \GMS  \to \bP \cat{Rub}$ to $\GMS(\Gamma)$ is a regular embedding.
\end{prop}
\par
\begin{proof}
We adapt the usual strategy to show that an embedding of a smooth scheme into an ambient smooth scheme is regular, compare \cite[Lemma~10.106.4]{stacks-project}. We tacitly fix a component $C$ of $\GMS(\Gamma)$. Let~$R$ be the (regular) local ring of $\bP \cat{Rub}$ in a neighborhood of a given point~$x$ of $\GMS(\Gamma)$. We aim to show that there is a (minimal) set of generators $f_1,\ldots,f_{3g-3+n}$ of the maximal ideal~$\frakm$ of~$R$ such that the ideal~$I$ of $\GMS(\Gamma)$ in~$R$ is cut out by $f_1,\ldots,f_c$ where $c = g-|V_C^\ab| + L(\Gamma)$. We let~$J$ be the ideal defining~$\GMS(\Gamma)^\red$ in~$R$. Let $\ol{\frakm} = \frakm/J$ be the maximal ideal in~$R/J$. Since
\ba \label{eq:jfrakm}
\dim_\bC((J+\frakm)/\frakm^2) &\= \dim_\bC \frakm/\frakm^2 - \dim_\bC \ol{\frakm}/\ol{\frakm}^2 \\
&\= \dim \bP \cat{Rub} - \dim \GMS(\Gamma)^\red \= g-|V_C^\ab| + L(\Gamma) \= c\,,
\ea
we may choose $f_1,\ldots,f_c \in J$ whose images in $\frakm/\frakm^2$ are linearly independent, and more precisely we may choose $f_i = t_i$ for $i=1,\ldots,L(\Gamma)$. Since $J = \langle I, t_1,\ldots,t_L \rangle$ by definition of $\GMS(\Gamma)^\red$ we may choose $f_i \in I$ for $i=L+1,\ldots,c$. We complement these generators by $f_{c+1},\ldots,f_{3g-3+n}$ to get a minimal system of generators of~$\frakm$. As in loc.\ cit.\ it follows that $J = \langle f_1,\ldots,f_c \rangle.$. Now
\be \label{eq:genofI}
I_0 \,:= \, \langle f_1^{\ell_1},\ldots,f_L^{\ell_K},f_{L+1}, \ldots, f_c \rangle \subseteq I.
\ee
To see that they are equal we need to use that the powers~$\ell_i$ of $t_i$ are minimal to define~$\GMS$. More precisely, suppose that $y \in I$ were an element that exhibits~\eqref{eq:genofI} to be strict. Since $y \in J$ we may without loss of generality assume
$$y \= \sum_{i=1}^L \sum_{j_1=1}^{\ell_1-1} \cdots \sum_{j_L=1}^{\ell_L-1} a_{i,\bfj} \prod_i t_i^{j_i}, \qquad \bfj = (j_1,\ldots,j_L)$$
with $a_{i,\bfj} \in \langle f_{L+1}, \ldots, f_c\rangle$ and some $a_{i,\bfj} \neq 0$ for some $(i,\bfj)$. We may extend the deformations of the twisted differentials at each level (prescribed by the $a_{i,\bfj}$) in a arbitrary way to order $t_i^{\ell_i}$ and glue the resulting twisted differentials as in the proofs in Section~\ref{sec:support} to a family over $R/I_0$ (using once again that the residue constraint disappears if we glue to order at most $t_i^{\ell_i}$). This shows that such an element~$y$ does not belong to~$I$.
\end{proof}
\par
From this proposition we deduce the following statement, stating that components of~$\GMS$ intersect 'as transversely as possible'. We state the first non-trivial case. The general case (with more levels and more components, see the ``Case: general graph~$\Gamma$'' in Section~\ref{sec:support}) can be dealt with similarly.
\par
\begin{prop} \label{prop:GMSequations}
Suppose that $\mu$ is a meromorphic signature and that $\Gamma$ is a graph with one level passage. Suppose moreover at $x \in \bP\cat{Rub}$ the component $C$ of $\GMS(\Gamma)$ and $\MS$ meet, and that $b = |V(\Gamma)^\ab_C|>0$. Then there exist functions $r_1,\ldots r_b$ and $f_{b+1},\ldots,f_g$in the local ring of $\bP\cat{Rub}$ at~$x$ such that
\be
I \= \langle r_1 t_1^{\ell_1},\ldots, r_bt_1^{\ell_1}, f_{b+1},\ldots,f_g \rangle
\ee
generates the ideal of $\GMS \subset \bP\cat{Rub}$ near~$x$, while 
$r_1 ,\ldots, r_b, f_{b+1},\ldots,f_g$ is a regular system for $\MS$ and
$t_1^{\ell_1}, f_{b+1},\ldots,f_g$ is a regular system for the component~$C$ of
$\GMS(\Gamma)$.
\par
In particular $\frakj_0: \GMS \to \bP\cat{Rub}$ is a regular embedding near~$x$ if and only if~$b=1$.
\end{prop}
\par
The restrictions to~$\GMS$ of the $r_i$ are the sums of residues at the edges to the various components as in the condition~$(\GRC_1)$. It would be interesting to have a geometric interpretation on the whole (Zariski neighborhood of~$x$ in) $\bP \cat{Rub}$.
\par
\begin{proof} Let $r_1,\ldots r_b$ be the residue sums that appear in~$(\GRC_1)$. These are functions on $\GMS$ near~$x$. Since $\GMS$ is a substack, these are restrictions of functions on $\bP \cat{Rub}$ that we denote by the same letter.
Since we suppose $\mu$ is a meromorphic signature these functions are indeed independent and part of the (``perturbed period'') coordinate system both on~$\MS$ and on~$\GMS(\Gamma)$.
\par
From Proposition~\ref{prop:GMSGammaregemb} we take the functions $f_i$ such that $t_1^{\ell_1}, f_{b+1},\ldots,f_g$ is a regular sequence for~$C$ and $t_1, f_{b+1},\ldots,f_g$ is a regular sequence for~$C^\red$. More precisely, since the tangent space to~$\GMS$ has dimension at most (and a posteriori equal to) $3g-3+n-g+b$ by Corollary~\ref{cor:tangentspacecontained}, we may choose (using~\eqref{eq:jfrakm}) the~$f_i$ as elements of the ideal of~$\GMS$ at~$x$. By the independence of the $r_i$ there exist functions $h_i$ such that
\be \label{eq:RSEQ}
t_1, f_{b+1},\ldots,f_g,r_1,\ldots r_b,h_{g+1},\ldots,h_{3g-3+n}
\ee
is a regular system for the maximal ideal at~$x$.
\par
Now we use that any permutation of the elements in~\eqref{eq:RSEQ} is also a regular sequence and we start with~$r_1,\ldots r_b,f_{b+1},\ldots,f_g$. By our choice of the~$f_i$ the space $\MS$ is contained in the vanishing locus of these~$g$ elements and a dimension check completes the claim that they form a regular system for~$\MS$.
\end{proof}


%

\subsection{Tangent bundles and the normal bundle} \label{sec:tangnormal}

We start with an approximate computation of the Chern classes tangent bundle to  the multi-scale space. Note that \cite{CMZEuler} gives a complete anwer to this question. However the two resulting formulas are quite different.\footnote{E.g., the formula given there involves on the small tautological ring without the horizontal divisor, whereas the formulas that follow from Proposition~\ref{prop:tangMS} involve $\lambda$-classes that a priori do not belong to this ring.} The one here is adapted to subsequently computing normal bundles.
\par
We write $\frakj: \MS \to \bP \cat{Rub}$ for the map~$\frakj_\Gamma$ in the special case where $\Gamma = \bullet$ is the trivial graph. We define $\LG_1^{\GRC} \subset \LG_1$ to be the set of level graphs that impose a non-trivial GRC on lower level. 
\par
\begin{prop} \label{prop:tangMS}
There is an exact sequence of coherent sheaves on $\MS$:
\be \label{eq:GMSseq}
0 \to \T {\MS} \xrightarrow{d \frakj}
\frakj^* \T \bP \cat{Rub} \xrightarrow{b}   \frakj^* \T_{\rel}\Picabs
\to \cC \to 0
\ee
where the cokernel~$\cC$ is supported on the divisors $D_{\Gamma,\bullet}$ for $\Gamma$ in  $\LG_1^{\GRC}$ if $\mu$ is a meromorphic signature or $k>1$. In the case that $k=1$ and $\mu$ is a holomorphic signature, the cokernel~$\cC$ is an extension of $\cO_{\MS}(-1)$ with a sheaf supported on $D_{\Gamma,\bullet}$ for $\Gamma$ in  $\LG_1^{\GRC}$.
\end{prop}
\par
\begin{proof} The injection of $\T \MS$ into the kernel of $b$ follows from the definition of $\MS$ as a component of the locus where~$\sigma_\cL$ and~$\sigma_0$ agree. If~$x$ is a not boundary point where the level graph~$\Gamma$ has a GRC, then the image of $b_\ab$ in Corollary~\ref{cor:tangentspacecontained} has dimension one in the holomorphic and zero in all other cases. This shows that the dimension of $\T \MS$ and of the kernel agree. Since $\MS$ is smooth this proves the exactness at $\frakj^* \T \bP \cat{Rub}$.
\par
In the meromorphic case we simply need to show that $d\sigma_\cL$ is generically surjective outside the union of $D_{\Gamma,\bullet}$ for $\Gamma$ in  $\LG_1^{\GRC}$. In the holomorphic case we consider  outside the union of those $D_{\Gamma,\bullet}$ the surjective map $\frakj^* \T_{\rel}\Picabs \to \cO_{\MS}(-1)$ induced by the dual to the inclusion~\eqref{eq:OmegaMaxincl}. It follows from Corollary~\ref{cor:tangentspacecontained} compared with the dimension of~$\MS$ that the image of $b$ there equals the kernel of this map, proving the remaining exactness.
\end{proof}
\par
For the case of a general component $\GMS(\Gamma)$ we first apply this argument to the top and bottom levels to $\cat{Rub}(\Gamma)$ and then deal with the normal directions outside the boundary.
\par
\begin{prop} \label{prop:tangGMSGamma}
There is an exact sequence of coherent sheaves on $\GMS(\Gamma)$
\be \label{eq:GMSseqGamma}
0 \to \T {\GMS(\Gamma)^\red}  \xrightarrow{d \frakj_\Gamma}
(\frakj_\Gamma^\red)^* \T \bP\cat{Rub}(\Gamma) \xrightarrow{b}   (\frakj_\Gamma^\red)^* \T_{\rel}\Picabs \to \cC_\Gamma \to 0\,,
\ee
where the cokernel~$\cC_\Gamma$  is an extension of $\oplus_{v \in V(\Gamma)_C^\ab} \cO_{\mu,v}(-1)$ with a sheaf supported on $D_{\Gamma,\bullet}$ for $\Gamma$ in  $\LG_1^{\GRC}$. Here $V(\Gamma)_C^\ab$ is determined by the component $C$ of $\GMS(\Gamma)$ as in \eqref{eq:VabC}.
\end{prop}
\par
\begin{proof} The idea is the same as for Proposition~\ref{prop:tangMS}. Again, outside the intersection with GRC-divisors $D_{\Gamma,\bullet}$ the map to $\oplus_{v \in V(\Gamma)^\ab} \cO_{\mu,v}(-1)$ is induced by the dual to inclusion of differential forms in Corollary~\ref{cor:tangentspacecontained}. The image of $d\frakj_\Gamma$ is obviously contained in the kernel of~$b$. To show equality we use the smoothness of $\GMS(\Gamma)^\red$ in Proposition~\ref{prop:GMSred} and compare dimensions: Observe that $\bP\cat{Rub}(\Gamma)$ has codimension~$L(\Gamma)$ in $\bP\cat{Rub}$, locally being cut out by $t_i=0$. The dimension claim now follows from the corollary combined with Lemma~\ref{le:dGammadef}.
\end{proof}
\par
We continue to denote by $\cC_\Gamma$ the cokernel as in the previous propositions. We define
\be \label{eq:defLtop}
\cL_{\Gamma}^\top \=  \cO_{D_\Gamma} \Bigl(\sum_{\wh{\Delta} \in \LG_2
\atop \delta_2(\wh{\Delta}) = \Gamma} \ell_{\wh{\Delta},1}D_{\wh{\Delta}} \Bigr)
\ee
to be the bundle summing all the degenerations of the top level with the factor
$\ell_{\wh{\Delta},1}$, the lcm of the top level passage enhancements.\footnote{This agrees with \cite[Equation~(49)]{CMZEuler}, not to be confused with the bundle in~\eqref{eq:defcL}.}
\par
\begin{theorem} \label{thm:nbGMSapprox}
The Chern polynomial of normal bundle  $\cN_{\frakj_\Gamma \circ \iota_\red}$ of the inclusion of $\GMS(\Gamma)^\red$ in $\bP \cat{Rub}$ can be computed as
\be
c(\cN_{\frakj_\Gamma \circ \iota_\red}) \= c(\ker((\frakj_\Gamma^\red)^* \T_{\rel}\Picabs \to \cC_\Gamma)) \cdot c(\cN_\Gamma) 
\ee
where $\cN_\Gamma$ is the normal bundle of $\bP \cat{Rub}(\Gamma)$ in $\bP \cat{Rub}$ restricted to $\GMS(\Gamma)$.  
\par
If $\Gamma$ has just one level passage, then $\cN_\Gamma$ is a line bundle whose Chern polynomial is given by
\be
c(\cN_\Gamma) = 1 + \frac{1}{\ell_\Gamma}(- \zeta^\bot + \zeta^\top  + \cL_\Gamma^\top) \,.
\ee
\end{theorem}
\par
\begin{proof}
The first statement follows from the description of $\T {\GMS(\Gamma)^\red}$ in the previous Propositions~\ref{prop:tangMS} and~\ref{prop:tangGMSGamma} and the definition of the normal bundles.
\par
The computation of the normal bundle $\cN_\Gamma$ is essentially the same as the computation of the normal bundle of a boundary divisor~$D_\Gamma$ inside $\MS$ in \cite[Theorem~7.1]{CMZEuler}: The normal bundle is locally generated by the level parameter~$t_i$. The $\ell$-th power of~$t_i$ is given (away from degenerations) as the ratio of the tautological section on top an bottom level, thus the term $- \zeta^\bot + \zeta^\top$. If the top level degenerates, this ratio statement is no longer true, requiring the correction by  $\cL_\Gamma^\top$.
See also the comments in \cite[Section~4.3]{CMSeclinsub} for the adaptation to $k$-differentials, using the conventions in~\eqref{eq:xizetaconv}. 
\end{proof}

%% file: sec_excess.tex
\section{The excess intersection classes that contribute to the volume}
\label{sec:excess}

We compute the contributions to $\logDR_\cL = \sigma_0^! \sigma_\cL(\bP\cat{Rub}_\cL)$ that pair non-trivially with $\eta^d$. The first step is to single out the relevant type of graphs. 
\par
\begin{definition}
A level graph~$\Gamma$ is called a \emph{($k$-)simple star graph}, if
\begin{enumerate}
\item it has one vertex on bottom level (the \emph{central vertex}) and
\item for each vertex on top level (the \emph{outlying vertices}) all the adjacent legs (the~$m_i$ at marked points and the enhancement~$\kappa_e$ at the edges) have orders divisible by~$k$.
\end{enumerate}
\end{definition}
Note that a simple star graph might be not of compact type, and there is no restriction to the genus and the markings at the central vertex.
\par
\begin{definition} \label{def:SSStar}
A simple star graph is called \emph{special} if
\begin{enumerate}
\item it is compact type,
\item the genus of the central vertex is zero,
\item the vertices on top level carry no markings,
\item and the central vertex carries precisely two markings (or just one marking, if $k=1$ and $\mu=(2g-2)$ is a single zero).
\end{enumerate}
\end{definition}
We denote these sets of graphs by $\SStar(\mu)$ and $\SSStar(\mu)$ respectively.
\par
\begin{definition} \label{def:sunflower}
A level graph~$\Gamma$ is called a \emph{sunflower graph}, if
\begin{enumerate}
\item it is compact type,
\item it has one vertex (the \emph{sun}) on top level, which is connected to every vertex on bottom level,
\item each bottom level vertex has genus zero and precisely one marking, and
\item for every vertex on top level besides the sun all the adjacent legs have orders ($m_i$ and $\kappa_e$) divisible by~$k$.
\end{enumerate}
\end{definition}
\par
We denote this set of graphs by $\SF(\mu)$. The reason for this definition is: 
\par
\begin{prop} \label{prop:xitopgraphs}
Suppose that $\alpha_\Gamma = \iota_{\Gamma,*} [\alpha] \in \CH^*(\GMS)$ is a pushforward from $\GMS(\Gamma)$ with $\Gamma \neq \bullet$ and that $\alpha \cdot \eta^d \neq 0$. Then $\Gamma$ is obtained from a special ($k$)-simple star graph or from a sunflower graph by possibly degenerating the lower level.
\par
If $\mu$ is of holomorphic abelian type, then more precisely such a graph~$\Gamma$ is a special simple star graph with a unique marking, i.e., necessariliy $\mu = (2g-2)$.
\end{prop}
\par
\begin{proof} Let $h_\na$ be the number of top level vertices of~$\Gamma$ that are not holomorphic abelian. We label the top level vertices, starting with the $h_\ab$ abelian ones. Each vertex carries a triple of information $g_i,n_i,m_i$ encodding its genus, the number of marked points and the number of nodes attached to it. We may merge the lower levels as this does not change neither hypothesis nor conclusion and suppose that~$\Gamma$ has just one level passage.
\par
We start with the case that $\mu$ is not a holomorphic abelian signature. By Lemma~\ref{le:etaISxitop} and the hypothesis $\alpha \cdot \eta^d \neq 0$ the projectivized top level dimension has to be at least $2g-3+n$, which implies
\begin{equation*}
    2g-3+n \,\leq \, \sum_{i=1}^{h_\ab}(2g_i-1+n_i+m_i) + \sum_{i=h_\ab+1}^{h_\ab+h_\na}(2g_i-2+n_i+m_i)-1
\end{equation*}
or equivalently that 
\begin{equation*}
2\left(g-\sum_{i=1}^{h_\ab+h_\na}g_i\right) +
\left(n-\sum_{i=1}^{h_\ab+h_\na}n_i\right)
\,\leq\,  2-h_\na +\sum_{i=1}^{h_\ab+h_\na}(m_i-1)\, .
\end{equation*}
\par
Note that the top $\zeta$-power of a holomorphic abelian top level vertex will be zero if it has more than one node. Therefore $m_i=1$ for $i=1,\ldots,h_\ab$, and we abbreviate by
\begin{eqnarray*}
r_{g} \ = g-\sum_{i=1}^{h_\ab+h_\na}g_i, \qquad 
r_{n} \=  n-\sum_{i=1}^{h_\ab+h_\na}n_i, \qquad
r_{m} \=  \sum_{i=h_\ab +1}^{h_\ab+h_\na}(m_i-1)
\end{eqnarray*}
the remaining genus, marking and nodes (for bottom level) respectively.
\par
We claim that $1 + r_m - r_g$ is a lower bound for the number of vertices on bottom level. This bound holds if~$\Gamma$ is a tree and continues to hold for every insertion of any further edges, thus justifying the claim.
\par
Moreover,  there should be at least one marking for every bottom vertex. Taken together
\begin{equation*}
    1+r_m-r_g \,\leq\, \#\{\text{bottom level vertices}\}\,\leq \,r_n \,\leq\, 2-h_\na + r_m -2r_g\, .
\end{equation*}
This leads to the inequality
\begin{equation*}
    0\,\leq\, 1-h_\na -r_g
\end{equation*}
which implies one of the following cases:
\begin{itemize}
    \item $r_g=1$ and $h_\na =0$: then the number of bottom level vertices is~$r_m=0$, which is impossible.
    \item $r_g=0$ and $h_\na=1$: then the number of bottom level vertices is $r_n=1+r_m$, so each contains exactly one marking and is connected to the unique non-holomorphic abelian top vertex, i.e., we have a sunflower graph.
    \item $r_g=0$ and $h_\na=0$: then there are $1+r_m$ bottom level vertices one with two markings and the rest with one marking each. However, as $r_m=0$, there is a unique bottom component containing two markings. This case only happens when $|\mu|=2$, i.e., we have a special simple star graph.
\end{itemize}
\par
If $k=1$ and $\mu$ is holomorphic, then the projectivized top level dimension has to be at least $2g-2+n$, which leads to the inequality
\begin{equation*}
    2g-2+n \,\leq \, \sum_{i=1}^{h_\ab}(2g_i-1+n_i+m_i) + \sum_{i=h_\ab+1}^{h_\ab+h_\na}(2g_i-2+n_i+m_i)-1\, .
\end{equation*}
The inequality on the number of bottom level vertices now becomes
\begin{equation*}
    1+r_m-r_g \,\leq\, \#\{\text{bottom level vertices}\}\,\leq \,r_n \,\leq\, 1-h_\na + r_m -2r_g\, .
\end{equation*}
Then $r_g=0$, $h_\na=0$, $r_m=0$, and $r_n=1$ which implies that there is a unique bottom component with genus zero and a single marking. This can only happen when $|\mu|=1$ and $h_\ab\geq 2$ due to stability on the bottom vertex, again a special simple star graph.
\end{proof}
\par
Inspection of the cases also shows in the situation of the preceding proposition:
\par
\begin{cor} \label{cor:topleveldim}
If $\alpha \cdot \eta^d \neq 0$ on $\GMS(\Gamma)$, then the dimension of the projectivized top level statum of~$\Gamma$ is equal to $d$ (both in the meromorphic case and in the holomorphic case, where we abuse $d$ instead of writing $d_\ab$.)
\end{cor}
\par
Next we observe that this graphs indeed are the only ones that pair non-trivially with the~$\eta^d$. Recall from~\eqref{eq:dGammadef} that $d_\Gamma = \dim(\GMS(\Gamma))$.
\par
\begin{prop} \label{prop:DRwithxitop}
  The intersection of $\logDR$ with the top power of $\eta$ is a sum
\be
[\logDR_\cL \cdot \eta^d] \=
\sum_{\Gamma \in \SF(\mu) \cup \SSStar(\mu) \atop \text{or $\Gamma = \bullet$}}
\iota_{\Gamma,*} [\DR_{d,\Gamma}]
\ee
over contributions from the trivial graph, sunflower graphs and special simple star graphs. Here
\be \label{eq:excessclass}
[\DR_{d,\Gamma}] \= \ell_\Gamma \cdot c_{d_\Gamma -d}(\frakj_\cL^* \cN_{\sigma_0} / \cN_{j_0|\GMS(\Gamma)^\red}) \cdot \eta^d\,.
\ee
\end{prop}
\par
\begin{proof} By definition of intersection products (e.g.\ \cite[Section~6.1]{Fulton}) $\logDR_\cL = \sum_{i} \ell_i \alpha_i$ where the sum is over the distinguished subvarieties of the intersection product, i.e. over the images of irreducible components~$C_i$ of the normal cone of $\GMS$ in $\bP \cat{Rub}$. Here $\ell_i$ is the geometric multiplicity of~$C_i$ in $C = C_{\cat{Rub}} \GMS$. By Proposition~\ref{prop:GMSunion} the distinguished subvarieties are the $\GMS(\Gamma)^\red$ for $\Gamma$ imposing a GRC at each of its levels, and intersections of such.
\par
Now we apply Proposition~\ref{prop:xitopgraphs}. We deduce that among the substacks $\GMS(\Gamma)^\red$ only the graphs~$\Gamma$ listed in the proposition give components of the normal cone where the intersection with $\eta^d$ may have a non-zero contribution irrespectively of the Segre class of the cone. (Observe that degenerating the lower level of $\Gamma \in \SF(\mu) \cup \SSStar(\mu)$ will not give a graph that imposes a GRC at all levels, so these degenerations do not show up in the list.)
\par
We next argue, that no distinguished subvariety contained in an intersection of two such $\GMS(\Gamma_i)^\red$ for $i=1,2$ pairs non-trivially with $\eta^d$. The level graph at such an intersection has to be a bottom level degeneration of (say~$\Gamma_1$) a sunflower or a special simple star graph by Proposition~\ref{prop:xitopgraphs}. But the graph~$\Gamma_2$ that arises from top level undegeneration the does not impose a GRC at all level, and thus does not determine a component of~$\GMS$, contradiction.
\par
Since the $\GMS(\Gamma)^\red$ are smooth hence regularly embedded, their contribution can be computed using the excess intersection formula as given. The final claim implicit in the notation is that the geometric multiplicity equals~$\ell_i$, which can be read off from the local equations for $\GMS(\Gamma)$ in Proposition~\ref{prop:GMSequations}.
\end{proof}
\par
\medskip
For any signature~$\mu$ (without imposing restrictions neither on the $m_i$ nor on~$k$) we define \emph{intersection-theoretic volume} of a stratum to be
\be
\vol(\mu) \,:=\, \vol(\komoduli[g,n](\mu)) \= \int_{\bP\komoduli[g,n](\mu)} \zeta^d
\ee
the top intersection of the dual tautological class~$\zeta$. We use $\vol(\mu)$ in small letters for the intersection-theory side and compare in Section~\ref{sec:completedvolumes} with Mazur-Veech volumes~$\Vol(\mu)$. Our goal here is to compute the \emph{DR-completed volume}
\be \label{eq:volDR}
\ol{\vol}(\mu)^\DR \= \int_{\barmoduli[g,n]} \DR(\mu) \zeta^d
\ee
in terms of intersection-theoretic volumes of strata.
\par
\begin{theorem} \label{thm:DRcompletedvolgraphs}
The DR-completed volume can be written as a sum
\be
\ol{\vol}^\DR(\mu) \=  {\vol}(\mu) + \sum_{\Gamma \in \SF(\mu) \cup \SSStar(\mu)} \frac{1}{\Aut(\Gamma)k^{h_\ab}} \prod_{e \in E(\Gamma)} \kappa_e \prod_{v \in V(\Gamma)} \vol(\mu_v)
\ee
over all special simple star graphs and sunflower graphs, where $\mu_v$ is the signature of all legs adjacent to the vertex~$v$, and $h_\ab$ is the number of holomorphic abelian top components of the corresponding graph $\Gamma$. 
\end{theorem}
\par
Here the volume of abelian top components has to be taken with respect to its tautological class $\zeta = -\xi$ as space of abelian differentials. If considered as a space of $k$-powers of abelian differentials, this introduces a factor~$k$ for each such vertex (compare with~\ref{eq:xizetaconv}), cancelling the $k^{h_\ab}$-denominator.
\par
\begin{proof} Evaluating the DR-volume on the left hand side we apply 
  Proposition~\ref{prop:xitopgraphs} to first reduce to the summands on the right hand side. For each of those summand we use Theorem~\ref{thm:nbGMSapprox} to evaluate that the Chern polynomial of the excess bundle is given by
  \be
c(\frakj_\cL^* \cN_{\sigma_0} / \cN_{j_0|\GMS(\Gamma)^\red}) \=  c(C_\Gamma) \cdot c(\cN_\Gamma)^{-1}
  \ee
since $\T_{\rel}\Picabs = \cN_{\sigma_0}$ with~$\cC_\Gamma$ as in Proposition~\ref{prop:tangGMSGamma}.
\par
We next argue that we can replace $c(C_\Gamma)$ by $\oplus_{v \in V(\Gamma)^\ab} \cO_{\mu,v}(-1)$ when taking the preceding equation $\cdot \eta^d$. In fact, there are two cases. First, consider $\Gamma = \bullet$, i.e., the component~$\MS$. In the meromorphic (or $k>1$ case) the excess class is a zeroth Chern class. In the holomorphic case the only graph that appears in Proposition~\ref{prop:xitopgraphs} has an inconvenient (in the sense of \cite{MUWrealize}) bottom vertex, i.e., does not appear in the boundary of $\MS$. Second, if $\Gamma$ is a special simple star or a sunflower. Then the botton level has no degenerations that impose GRC's as neither the level-moving in the sunflower case does, nor such exist for the star graphs where there is at most one zero and the rest are (residue-unconstrained) poles. Any top level degeneration of such a~$\Gamma$ is not a special simple star nor a sunflower, so top-$\zeta^d$ evaluates to zero there whatever the excess class may contribute.
\par
From Corollary~\ref{cor:topleveldim} we deduce that the full $d_\Gamma-d$-th Chern class taken from the excess class must be supported on lower level. Since the remaining contributions from~$\cC$ are to the upper level, this term must come from $c(N_\Gamma)^{-1}$, and in fact the only part of the excess class that contributes equals $\zeta^{d-d_\Gamma}$. This evaluates to the product of the volumes on the lower level vertices. A standard Segre class computation evaluates the term~$\zeta^d$ as the product of the volumes of the top level vertices.
\par
In the preceding conversions to volumes we have been using the level-wise~$\zeta^\top$. The conversion to the (global) $\zeta$ appearing in~\eqref{eq:volDR} is responsible for the factor  $\prod \kappa_e/|\Aut(\Gamma)|\ell_\Gamma$ by Corollary~\ref{cor:xiconversion}. The $k$-power arise from the conversion of~$\xi$ to~$\zeta$.
\end{proof}

%% file: sec_volumes.tex
\section{Volumes of strata and DR-volumes}  \label{sec:completedvolumes}

For any stratum in $k=1$ or $k=2$ and of finite volume (i.e., $m_i >-k$ or equivalently $a_i >0$ for all~$i$) we denote by $\Vol(\mu) = \Vol(\komoduli[g,n](\mu))$ the \emph{Masur-Veech volume} in the normalization as explained in \cite{CMSprincipal}, taking the hypersurface volume on differentials with unitary area.
If the stratum is \emph{\textrm{REL} zero}, i.e., if $k=1$ and $n=1$ or $k=2$ and all $m_i$ are odd, then \cite{SauvagetMinimal} and \cite{CMSprincipal} relate the Masur-Veech volume and the intersection theoretic volume by
\begin{eqnarray*}
 \Vol(\mu) & \= & \frac{2(2\pi i)^{2g}}{(2g-1)!} \vol(\mu) \ \  \text{for} \ \  k=1, \  n=1\,,\\
 \Vol(\mu) & \= & \frac{2^{-n+3} (2\pi i)^{2g-2+n}}{(2g-3+n)!} \vol(\mu) \ \text{for} \  k=2, \ m_i's \ \text{odd}\,.
\end{eqnarray*}
\par
In the case $k=2$ and all entries $m_i>-2$ odd the \emph{completed volume} by $\ol{\Vol}(\cQ(\mu))$ of strata was defined in \cite{DGY} as an approximation to the Masur-Veech volume using counts of integer metrics on Ribbon graphs via Kontsevich polynomials. This completed volume is expressed as the volume of the strata plus the contribution of adjacent strata coming from the counting of integer metrics on degenerate ribbon graphs. The goal of this section is to show that this completed volume can also be interpreted as a sum over contributions from level graphs. In order to state the formula in parallel to the DR-case we define the (Ribbon graph) completed volume in the intersection theory convention by
\be \label{eq:Volvol}
\ol{\Vol}(\mu) \= \frac{2^{2g+1}(-1)^{g-1+\frac{n}{2}} \pi^{2g-2+n}}{(2g-3+n)!} \ol{\vol}(\mu)\,
\ee
where $n$ is even as all the entries $m_i$ are odd and their sum is $2g-2$.
\par
We will prove two theorems. The first theorem for  $|\mu|\geq 3$ requires simply a matching of terms and of volume conventions. The second theorem for $|\mu|=2$ requires a complete discussion of ribbon graphs and their completed volumes, since this case was omitted in \cite{DGY}. Observe that in the case $|\mu| \geq 3$ only sunflowers appear in Proposition~\ref{prop:xitopgraphs}. This fact is matched by the following formula.
\par
\begin{theorem} \label{thm:completedvolgraphs3plus}
For $k=2$ and $|\mu|\geq 3$ the Ribbon-graph completed volume can be written in intersection theory convention as a sum
\be \label{eq:Qvol3}
\ol{\vol}(\mu) \=  {\vol}(\mu) + \sum_{\Gamma \in   \SF(\mu)} \frac{1}{\Aut(\Gamma)2^{h_\ab}} \prod_{e \in E(\Gamma)} \kappa_e \prod_{v \in V(\Gamma)} \vol(\mu_v)
\ee
over all sunflower graphs, where $\mu_v$ is the signature of all legs adjacent to the vertex~$v$. 
\end{theorem}
\par
Observe that the last product is over all vertices of~$\Gamma$, including those on bottom level, which have higher order poles and where the volume has known no Masur-Veech-type interpretation.
\par
\begin{proof}
 The formula presented in  of \cite[Theorem~2.6]{DGY} decomposes the completed volume for a stratum $\mu=(m_1,\ldots,m_n)$ with odd entrees with $n\geq 3$ and $m_i\geq -1$, is
\bas\
&\ol{\Vol}(\cQ(\mu)) -  {\Vol}(\cQ(\mu)) \\
\=   & \sum_{\substack{\bfg=(g_1,\ldots,g_n) \\ 0 \leq g_i \leq \frac{m_i+1}{4}}} \,\,\sum_{\substack{\text{$i$ such that} \\ g_i>0}} \,\,\sum_{\substack{\bfg_i=(g_i^{(1)},g_i^{(2)},\ldots)\\ |\bfg_i|=g_i, \ g_i^{(j)}>0}} C_{\bfg,\bfg_i}\Vol\Bigl(\cQ(\mu-4\bfg)\times \prod_{i,j}\cH(2g_i^{(j)}-2)\Bigr) \, .
    \eas
where the coefficient~$C_{\bfg,\bfg_i}$ is defined in~\eqref{eq:defCgg} below and where $\mu-4\bfg$ denotes vector subtraction.
\par
First, claim that the decomposition in this sum is in bijection with a sum over sunflower graphs. In fact, the disconnected stratum $\cQ(\mu-4\bfg)\times \prod_{i,j}\cH(2g_i^{(j)}-2)$ describes the top components of two level graph $\Gamma$, where $\cQ(\mu-4\bfg)$ corresponds to the non-abelian component (the \textit{sun}), and each $\cH(2g_i^{(j)}-2)$ is an abelian holomorphic principal stratum of genus $g_i^{(j)}$. We get a bottom component of genus zero for every $i$ such that $g_i>0$, where $g_i$ equals the sum of the genera of abelian vertices attached to that bottom compoenent. This bottom component has a unique marking $m_i$ and one nodes connecting it to the sun and one node connecting to each abelian component $\cH(2g_i^{(j)}-2)$ for $j=1,\ldots,|\bfg_i|$. Moreover, the intersection theoretic volume of this bottom vertex $v_i$ is
\be
 \vol(\mu_{v_i}) \= \frac{m_i!!}{(m_i-2(|\bfg_i|-1))!!}
\ee
as computed in Proposition~5.1 of \cite{CGPT25}.
\par
Second, we need to justify that the coefficient given in \cite{DGY} as
\be \label{eq:defCgg}
C_{\bfg,\bfg_i} \= \prod_{\substack{\text{$i$ such that}\\
g_i>0}}\frac{m_i!!}{(m_i-2(|\bfg_i|-1))!!}\frac{(m_i-4g_i+2)}{|\bfg_i|!2^{|\bfg_i|}}\frac{|\bfg_i|!}{|\Aut(\bfg_i)|}\prod_{j=1}^{|\bfg_i|}(2g_i^{(j)}-1)
\ee
in fact equals
\be
C_{\bfg,\bfg_i} \= \frac{1}{2^{2h_\ab}|\Aut(\Gamma)|} \prod_{e \in E(\Gamma)} \kappa_e \prod_{v \in V(\Gamma^{\perp})}\vol(\mu_v)\,,
\ee
where $\Aut(\bfg_i)$ is the subgroup of the symmetric group that permutes only places with the same $g_i$. In fact, the automorphisms of a sunflower graph have the property that
\bes
   \frac{1}{2^{h_\ab}|\Aut(\Gamma)|} \=  \prod_{\substack{\text{$i$ such that} \\
    g_i>0}} \frac{1}{2^{|\bfg_i|}|\Aut(\bfg_i)|} \, . 
\ees
Grouping the prongs as those attached to the sun and those attached to the remaining vertices we observe
\bes
    \frac{\prod_{e \in E(\Gamma)} \kappa_e}{2^{h_\ab}} \=  \prod_{\substack{i \  s.t. \\
    g_i>0}}(m_i-4g_i+2)\prod_{j=1}^{|\bfg_i|}(2g_i^{(j)}-1)\,.
\ees
\par
Third, the volume of the top level product stratum is given in \cite[Equation~(32)]{DGY}. Note however that they they normalize the Masur-Veech volume on strata of quadratic differentials with (the cone under) area~$1/2$. Converting all these volumes into our (unit volume) convention we obtain that
\be
\Vol\Bigl(\cQ(\mu')\times \prod_{i=1}^r\cH(2g_i-2)\Bigr) \= \frac{(d'-1)!\Vol(\cQ(\mu'))\prod_{i=1}^r(2g_i-1)!\Vol\cH(2g_i-2)}{(d-1)!}
\ee
where $d'=\dim\cQ(\mu')$, $2g_i=\dim\cH(2g_i-2)$, $d=\dim\left(\cQ(\mu')\times \prod_{i=1}^r\cH(2g_i-2)\right)$. Let $g'=\sum_{g_i\in \underline{g}} g_i$. Using the volume conversions
\begin{eqnarray*}
    \Vol(\cQ(\mu-4\bfg)) & = & \frac{2^{2(g-g')+1}(-1)^{(g-g')-1+\frac{n}{2}} \pi^{2(g-g')-2+n}}{(2(g-g')-3+n)!} \vol(\cQ(\mu-4\bfg))\,\\
    \Vol(\cH(2g_i^{(j)}-2)) & = & \frac{2(2\pi)^{2g_i^{(j)}}(-1)^{g_i^{(j)}}}{(2g_i^{(j)}-1)!}\vol(\cH(2g_i^{(j)}-2))\, .
\end{eqnarray*}
and combining the formula, we conclude the formula~\eqref{eq:Qvol3} in the theorem.
\end{proof}
\par
The goal of the rest of the section is:
\begin{theorem} \label{thm:completedvolgraphs2}
For $k=2$ and $|\mu|\geq 2$ the statement of Theorem~\ref{thm:completedvolgraphs3plus} also holds with the sum on the right hand side replaced by $\Gamma \in \SF(\mu) \cup \SSStar(\mu)$, now also including special simple star graphs.
\end{theorem}

\subsection{Ribbon graph counting}  \label{sec:ribbon}

We summarize the counting procedure of \cite{DGY}, define the various types of graphs appearing in their strategy, and show how they relate to the level graphs that appeared in the previous sections. 
\par
Associated with any square-tiled (abelian, or here rather) quadratic differential~$(X,q)$ there is
\begin{itemize}
\item the \emph{(cylinder) stable graph $\Delta$},\footnote{these are called~$\Gamma$ in \cite{DGY}, we try to reserve that letter for stable graphs that appear as level graphs},  and 
\item a ribbon graph~$G$, whose vertices are the singularities of~$q$, and whose boundary components in bijection with the cylinders, i.e.\ with the edges of~$\Delta$, and
\item an integer metric on (the edges of)~$G$, and heights of the cylinders.
\end{itemize}
\par
Here the graph~$\Delta$ is obtained by taking the limit point in the Deligne-Mumford compactification of the geodesic flow applied to $(X,q)$, equivalently the stable curve obtained by shrinking each horizontal cylinder of~$(X,q)$ to a point. The ribbon graph~$G$ is a neighborhood of the complement of the horizontal cylinders. 
\par
As in \cite{DGY} we often pass to the \emph{dual ribbon graph} of~$G$, denoted by~$G^*$. Its vertices now correspond to cylinders (and thus inherit the auxiliary labeling) and the faces correspond to singularities of~$q$.
\par
To count square-tiled surfaces with~$N$ squares (say) the main challenge is to count the number of integer metrics that can be assigned to the ribbon graph~$G$. For this purpose the boundary components of~$G$ are given an auxiliary labeling, and the goal is to count the number $F_G(\bfb)$ of integer metrics with boundary lengths~$\bfb = (b_1,\ldots,b_{2c}) \in \bZ_{>0}^{2c}$, where $c$ is the number of cylinders. Equivalently, we can count for each dual ribbon graph~$G^*$ the number $F_{G^*}(\bfb)$ of integer metrics with vertex parameters~$\bfb$. \cite[Section~3.3]{DGY} recalls properties of the function~$\vp$ that assigns the vertex paramters~$\bfb$ to a given tuples of edge lengths $w(e)_{e \in E(G^*)}$ of a (dual) Ribbon graph~$G^*$.
\par
\begin{lemma}
A dual ribbon graph~$G^*$ together with a pairing of vertices, i.e.\ a fixed-point free involution~$W$ on $\{1,\ldots,2c\}$ determines a cylinder stable graph~$\Delta$ with $c=E(\Delta)$ edges.
\end{lemma}
\par
The notation for the involution suggests that we are interested in the count of~$F_{G^*}(\bfb)$ for $\bfb$ lying on a \emph{wall}, defined as
\be \label{eq:wall}
W \= \{b_i - b_{W(i)} = 0\,\, \text{for all $i=1,\ldots,2c$}\} \subseteq \bZ_{>0}^{2c}\,.
\ee
intentionally abusing the letter~$W$.
\par
The strategy can now be summarized as follows:
\par
\begin{itemize}
\item[i)] The volume of the stratum~$Q(\mu)$ is the sum over all graphs~$\Delta$ -- more precisely over all dual ribbon graphs~$G^*$ with vertex pairing~$W$ that induce~$\Delta$ -- of the large~$N$-asymptotics of $F_{G^*}(\bfb)$ times height and width parameters that give in total given~$N$ squares (\cite[Proposition~4.1]{DGY}).
\item[ii)] The analagous statements holds for abelian strata with a single zero, and products of those strata (\cite[Section~4.3]{DGY}), which appear in the decomposition below. Both i) and ii) still hold for height and width parameters in a coset of finite index (\cite[Lemma~7.1]{DGY}).
\item[iii)] The sum over all~$G^*$ of $F_{G^*}(\bfb)$ is a quasi-polynomial, in fact a polynomial on every coset of $(2\bZ)^{2c} \subset \bZ^{2c}$, whose top term agrees with the Kontsevich polynomials $N_{g,2c}^\mu(\bfb)/2$ (\cite[Proposition~3.9]{DGY}).
\item[iv)] The Kontsevich polynomials $N_{g,2c}^\mu(\bfb)$ can be computed as intersection number of Witten-Kontsevich's combinatorial classes (\cite[Section~2 and Appendix~B]{DGY} and Section~\ref{sec:examples}).
\end{itemize}
However, since we need to evaluate the function $F_{G^*}(\bfb)$ on walls, the agreement with the Kontsevich polynomial in iii) no longer holds. This is the context of \cite[Section~5]{DGY}).
\par
We thus \emph{fix the cylinder stable graph~$\Delta$} with a \emph{fixed auxiliary labeling} of the non-leg half-edges $\{1,\ldots,2c\}$ with the involution~$W$ that pairs them. Our goal is an analog of \cite[Theorem~5.1]{DGY}, which makes explicit the overcounting of the Kontsevich polynomial compared to the metrices of all (dual) Ribbon graphs that give~$\Delta$ when paired with~$W$.\footnote{We aim for this overcounting theorem only in the special case of the walls we need. In loc.\ cit. more general wall structures are allowed, with the aim of e.g.\ proving (e.g. ``string'') equations for Witten-Kontsevich's combinatorial classes.}
\par
As explained in \cite[Proof of Theorem~5.1]{DGY} the overcounting stems from integral metrics on all dual ribbon graphs~$G^*$ that are compatible with $(\Delta,W)$ and where the edge length~$w_e = 0$ for a subset of $e \in E(G^*)$ that are called~\emph{degenerating edges}.
\par
Recall that the (global) Ribbon graphs~$G^*$ are non-bipartite, since we consider quadratic differentials that are not squares of an abelian differential. Call $e \in E(G^*)$ \emph{static}, if a connected component of~$G^*-e$ is bipartite.
The weight~$w(e)$ of an edge is uniquely determined by the vertex parameters~$\bfb$ if and only if~$e$ is static, combine \cite[Lemma~3.3 and Lemma~A.4]{DGY}. As a consequence:
\par
\begin{lemma}
The overcounting of a degenerate (dual) Ribbon graph contributes to the top order term of $F_{G^*}(\bfb)$ if and only if all the degenerating edges are static.
\end{lemma}
\par
The more precise goal is thus to describe all possible configurations of static degenerating edges on dual Ribbon graphs compatible with $(\Delta,W)$ and, for given splitting pieces, count the number of ways to reassamble them to a graph~$G^*$.
\par
Let $G^*_1,\ldots,G_h^*$ be the connected components of~$G^*$ minus the static edges $E_{\stat}(G^*)$. Then the \emph{static graph $G_{\stat}^*$} of~$G^*$, the graph with vertices $G_i^*$ and edges $E_{\stat}(G^*)$.
\par
\medskip
\paragraph{\textbf{Ribbon graphs to sunflower graphs}}  If all the static edges of~$G^*$ are bridge edges (i.e.\ they disconnect~$G^*$), then we may work with the consequences of \cite[Lemma~5.4]{DGY} as in the case of~$|\mu| \geq 3$, where this the only possibility that occurs.
\par
For the convenience of the reader, we explain how associate with~$G^*$ a sunflower graph~$\Gamma$, implicitly hidden in the proof of Theorem~\ref{thm:completedvolgraphs3plus}: the unique non-bipartite component~$G_0^*$ of~$G^*$ will be the sun of~$\Gamma$. Each face of $G_0^*$ without adjacent trees gives a marking (a leg) at the sun. All the other faces~$F$ of $G_0^*$ give a bottom level (genus zero) vertex of~$\Gamma$ with a unique marking of order corresponding to the number of edges adjacent to that face. For each component~$G_i^*$ that belongs to a subtree growing from a face~$F$ of~$G_0^*$ as add a top level vertex to~$\Gamma$ of the corresponding genus and adjacent precisely to the bottom level vertex corresponding to~$F$. Also in the case $|\mu| = 2$, with this assignment the ribbon graphs all whose static edges are bridges gives precisely the contribution of sunflower graphs in Theorem~\ref{thm:completedvolgraphs2}.
\par
In the remainder of the section we consider the complementary case and in particular \emph{restrict to the case $n=|\mu| = 2$.} The following replaces \cite[Lemma~5.4, Lemma~5.5 and~5.6]{DGY}.
\par
\begin{prop} \label{prop:Gstat}
Suppose that $G^*$ has a static non-bridge edge~$e$. Then the static graph~$G_{\stat}^*$ contains a unique ('main') cycle (consisting of static non-bridge edges).
\par
Each graph~$G^*_i$ is bipartite and has exactly one face. In particular $G_{\stat}^*$ naturally acquires the structure of a dual of a (planar) Ribbon graph~$G_{\stat}$.
\par
Moreover, the vertices of each index pair $(i,W(i))$ belong to the same component $G^*_{j(i)} = G^*_{j(W(i))}$ of the static graph.
\end{prop}
\par
\par
\begin{proof} If $e_1$ and $e_2$ are two static bridge edges, then $G^* - e_1 -e_2$ is disconnected (if not, any bicoloring of it leads to a contradiction when inserting the edges). The same argument shows that if there are three static bridge edges, then one of the paths joining the edges has to pass through~$e_2$. Inductively this shows that all static bridge edges lie on a cycle. The remaining (non-bridge) static edges must then grow from this cycles, by the very definition of 'non-bridge', giving $G_{\stat}^*$ the structure of one cycle with dangling trees.
\par
To show that each $G_i^*$ has exactly one face we argue inductively as in \cite[Lemma~5.5]{DGY} with the leaves of the dangling trees. It thus suffices to prove this property in case $G_{\stat}^*$, when the argument is similar: Consider a connected component~$G_i^*$. The two faces adjacent to a static edge~$e$ that enters the same cannot be equal, since otherwise~$e$ would be non-bridge. Joining the two half-edges that enter~$e$ we see that~$G^*$ has at least one more face than~$G_i^*$. Since~$G^*$ has just two of them, the claim follows.
\par
For the last claim and the components of the dangling tree the argument is as in \cite[Lemma~5.6]{DGY}: Consider a $G_i^*$ which is a leaf. Since the adjacent edge is degenerating and since $G_i^*$ is bicolored in black~$B$ and white $W$ vertices, we must have a relation $\sum_{i \in B} b_i = \sum_{i \in W} b_i$ by \cite[Lemma~3.3]{DGY}. Such a relating can follow from the defining relations~\eqref{eq:wall} only if the claimed condition holds for $G_i^*$. Now induct on the dimension to the leaf.
\par
We thus may assume that $G_{\stat}^*$ consists of the cycle only. The claim is trivial if this graph has only one edge, so assume it has two, denoted by~$e_1$ and $e_2$. Color $G^* - e_1$ so that this edge is adjacent to two black~$B$ vertices. The set of black and white vertices splits into two subsets for each of the two components of $G^*-e_1-e_2$, call them $B^\top, B^\bot, W^\top, W^\bot$ respectively, which are the two components $G_1^*$ and $G_2^*$ in this case. Now \cite[Lemma~3.4]{DGY} implies
\bas
0  &\=  w(e_1) \= \sum_{i \in B^\top} b_i + \sum_{i \in B^\top} b_i  - \sum_{i \in W^\top} b_i - \sum_{i \in W^\top} b_i  \\
0  &\=  \pm w(e_2) \= \sum_{i \in B^\top} b_i - \sum_{i \in B^\top} b_i  - \sum_{i \in W^\top} b_i + \sum_{i \in W^\top} b_i  \\
\eas
since coloring with respect to~$e_2$ flips the color on one of the two components. Adding and subtracting these equations gives equations that can only be deduced from~\eqref{eq:wall} if the claimed conditions hold. The general case follows by applying this argument to two adjacent edges of~$G_{\stat}^*$.
\end{proof}
\par
\medskip
\paragraph{\textbf{Ribbon graphs to special star-shaped graphs}} We associate to $G^*$ with a static non-bridge the special star-shaped graph~$\Gamma$ as follows. By definition there is precisely one bottom level vertex with the two markings and for each component $G_i^*$ (say with $e_i$ edges) of $G_{\stat}^*$ we add a top level holomorphic vertex without markings of genus~$g_i = (e_i+1)/2$ to~$\Gamma$ (and enhancement $\kappa_i = 2e_i = 4g_i-2$).
\par
\medskip
Next we analyze the converse step, namely how many ways there are to glue the given splitting pieces as in Proposition~\ref{prop:Gstat} to a given ribbon graph.
Suppose we are given
\begin{itemize}
\item[a)] bipartite dual ribbon graphs $G^*_1,\ldots,G_h^*$, each with one face, corresponding the minimal abelian stratum in genus~$g_i$, i.e.\ with $e_i := 2g_i-1$ edges and thus a face adjacent to $\kappa_i = 2e_i$ edges.
\item[b)] a vertex labeling and a fixed-point free involution, such that $(i,W(i))$ belong to the same component $G_j^*$ of the dual ribbon graph.
\end{itemize}
\par
We aim for Proposition~\ref{prop:GforgivenGi}, which is the analog of \cite[Proposition~5.7]{DGY}. We first compute the volume of the bottom level strata, to make the goal of the computation apparent. For $a \geq 2$ set $f_2(a,1) = 1/(a+2)$ and for $n \geq 2$ we let
\be
f_2(a,n) \= \frac{a!!}{(a-2(n-2))!!}
\ee
\par
This function appears in volume computations in the following form:
\par
\begin{prop} \label{prop:vol2zero}
The intersection-theoretic volume of the quadratic stratum in for the genus zero signature $\wt\mu = (m_1,m_2,-4g_1,\ldots -4g_{h^\ab})$ is
\be \label{eq:volmutilde}
\vol(\wt\mu) \= \sum_{I \subset \{1,\ldots, h^\ab\} \atop \text{such that
${c_{1,I}> 0}$}}
c_{1,I} f_2(m_1,|I|+1)f_2(m_2,|I^c|+1)
\ee
where $c_{1,I} = m_1 + 2 - \sum_{i \in I} 4g_i$.
\end{prop}
\par
\begin{proof}
The formula is the last (very long) displayed formula in the proof of \cite[Theorem~1.2]{CGPT25}.
\end{proof}
\par
To see which static graphs can appear with the given data, we fix a specific degeneration $\bfb'(\epsilon) \to \bfb$ for $\epsilon \in [0,1]$ starting at a point in the complement of all hyperplanes in the wall~$W$. That is we relabel the vertices such that $W(i) = i + c$, such that the labels $\leq c$ correspond to black vertices and the labels $\geq c+1$ are white vertices. Moreover we suppose that
\be \label{eq:biprime}
b_{i+c}' - b_i' = 2^i \epsilon\,,
\ee
i.e., the last tuple has the largest deviation from zero. While a priori the count might depend on the chosen order of equations, it turns out that it actually does not.
\par
\begin{prop} \label{prop:GforgivenGi}
The number~$\mathcal{G}_h = R(G_1^*,\ldots,G_h^*,\bfb')$ of ribbon graphs~$G$ whose degenerating edges as $\bfb'(\epsilon) \to \bfb$ form a static graph with pieces as in a), b) and which has precisely one cycle is equal to
\be 
\mathcal{G}_h \= \vol(\wt{\mu}) \cdot \prod_{e \in E(\Gamma)} \frac{\kappa_e}{2}
\ee
with $\vol(\wt{\mu})$ as in~\eqref{eq:volmutilde}
\end{prop}
\par
As part of the proof we describe the possibilities to make out of the vertices $G_i^*$ a graph $G_{\stat}^*$ with~$m$ edges, hence exactly one loop, and two faces, and such that moreover those faces are adjacent to $m_1+2$ and $m_2+2$ many edges respectively. We call this procedure a \emph{joining}, and subsequently count those joinings. Such a joining is called \emph{admissible} if there is a degeneration $\bfb'(\epsilon) \to \bfb$ such that the weight~$w(e)$ of every static edge is positive for $\epsilon > 0$.
\par
The ordering of the vertices such that~\eqref{eq:biprime} holds induces a natural ordering of the $G_i^*$, namely by the maximum vertex label~$s$ contained in $G_i^*$. For convenience we assume in the sequel that this order is increasing, i.e.\ that the $G_i^*$ are labeled such that
\be
\max\{ s: b_s \in G_i^*\} < \max\{ s: b_s \in G_j^*\}\quad \text{if and only if $i<j$.} 
\ee
We root the subtrees in $G_{\stat}^*$ at the point where they join the unique cycle and call \emph{descendents} of $G_i^*$ the vertices not on the path to the root.
\par
\begin{lemma}\label{le:coloring}
A joining is weakly admissible if and only if the following four conditions hold:
\begin{itemize}
\item Let $G^*_{\max}$ be the vertex of main cycle of $G_{\stat}^*$ that contains the largest wall equation inside or whose branch contains this wall equation. Then the two non-bridge degenerating edges adjacent to~$G^*_{\max}$ end at vertices of the same color. Define this color to be black.
\item For all other pieces of the cycle the two adjacent non-bridge degenerating edges end at vertices of different color.
\item If all descendants of~$G_i^*$ have labels smaller than~$i$, then the path from $G_i^*$ to its root emanates at a black corner of $G_i^*$.
\item Otherwise the path from $G_i^*$ to its root and the path to the subtree containing the descendent of maximal label are at vertices of different colors of~$G_i^*$.
\end{itemize}
A joining is admissible if the two faces of the resulting graph~$G^*$ are adjacent to $m_1+2$ and $m_2+2$ edges.
\end{lemma}
\par
\begin{proof}
  Consider how the edge lengths are computed from the vertex parameters \cite[Lemma~3.3]{DGY}, given that these are uniquely determined for static edges by \cite[Lemma~3.2]{DGY}. The last two statements can be derived exactly as in \cite[Lemma~6.2]{DGY}.
\end{proof}
\par
\begin{proof}[Proof of Proposition~\ref{prop:GforgivenGi}] Given an admissible joining with the induced coloring for each vertex $G_i^*$, whose face has $\kappa_i=4g_i-2$ edges, the coloring rules of Lemma~\ref{le:coloring} only allow to rotate all the attached graphs in steps of two, preserving the coloring. This gives $e_i=\kappa_i/2$ choices and we isolate this rotation factor $R=\prod_{i=1}^h e_i$.
  \par
  The graph $G_{\stat}^*$ has two faces corresponding to the zeroes $z_1,z_2$. The vertices $G_i^*$'s are divided into three disjoint sets: those fully adjacent to the face $z_1$ that we will index by $I$, those fully adjacent to $z_2$ that we index by $J$, and those adjacent to both faces forming the disconecting cycle that we will index by $K$. The perimeter of the disconnecting cycle adjacent to the face $z_1$ has length $c_{1,I}= m_1+2-\sum_{i\in I}4g_i$, and that adjacent to the face $z_2$ has length $c_{2,J}= m_2+2-\sum_{j\in J}4g_j$. Then, for such a $G_{\stat}^*$ to exist, it must be satisfied that $c_{1,I}, \ c_{2,J}>0$ and $K\not= \emptyset$. We denote by $\mathcal{N}(I,J,K)$ the number of such admisible joinings divided by the rotation factor~$R$. Then the total number of admisible joinings is the sum
\be\nonumber
    \mathcal{G}_h \= R \cdot \sum_{\substack{I\sqcup J \sqcup K\=\{1,.\ldots,h\} \\ c_{1,I}, \ c_{2,J}>0 \\ K\not= \emptyset}} \mathcal{N}(I,J,K)\, .
\ee
We proceed to compute the number $\mathcal{N}(I,J,K)$. Assume $G^*_{h}=G^*_{\max}$. The vertices indexed by $I$ form a (possibly disconnected) tree. The number of these trees is given as a factor of equation~(38) of \cite[Theorem~5.1]{DGY}, which we denote by the function $f_2(m_1,|I|+1)$. Analogously, the number of trees formed by the vertices indexed by $J$ is given by $f_2(m_2,|J|+1)$. The ways to attach these pair of trees to the main disconnecting cycle depends on the position of $G_h^*$. If $h\not\in I$, there are $c_{1,I}$ ways to attach the tree to main cycle. Analogously, if $h\not\in J$, there are $c_{2,J}$ ways to attach the corresponding tree. Otherwise, the ways to attach these trees, and the ways to attach the vertices indexed by $K$ in the main cycle, is given by the following two cases:
\par
\medskip
\paragraph{\textbf{Case 1: Suppose $h\in I$ or $h \in J$}} The coloring rules give a special role to the vertex to whom the branch containing $G_h^*$ is attached. We call it the \textit{special vertex} $G_t^*$ and reserve the symbol $t\in K$ for it. Fix an order of the vertices in the cycle. By the first coloring rule on Lemma~\ref{le:coloring}, the number of edges of $G_t^*$ facing the $z_1$-face is even, say $2l_t$. By the third and fourth coloring rules, the branch containing $G_h^*$ must be attached to one of the $l_t$ white corners of $G_t^*$.  For the rest of the vertices $G_i^*$ with $i\in K/\{t\}$, the second coloring rule says that the number of edges is odd, say $2l_i-1$. Then fix a tuple $\{l_i\}_{i\in K}$ with $1\leq l_i\leq e_i$ and $\sum_{i\in K}l_i=(c_{1,I}-1)/2$. By the coloring rule the number of ways to join a given tree formed out of the vertices of~$I$ is equal to~$l_i$, if $i=t$ is the special vertex. Then the total number of ways (over all $t\in K$) to join the labeled vertices in~$K$ up to rotation with a given $I$-vertex tree becomes
    \bas
    \mathcal{K}_{cycle} \= & \left(\sum_{i\in K}l_i\right) \#\left\{\{l_i\}_{i\in K} \ | \ 1\leq l_i\leq e_i, \ \text{and} \ \sum_{i\in K}l_i=(c_{1,I}-1)/2\right\}\\
    \= & \frac{c_{1,I}-1}{2}\left[t^{\frac{c_{1,I}-1}{2}}\right]\prod_{i\in K}(t+\ldots+t^{e_i})\\
    \= & \frac{c_{1,I}-1}{2}\sideset{}{^*}\sum_{L\subseteq K} \binom{(c_{1,I}-3-D_L)/2}{|K|-1}
    \eas
where $D_L=\sum_{i \in L}\kappa_i$ and where the star restricts the sum to terms where the numerator of the binomial coefficient is not smaller than the denominator. To rewrite the expression using $c_{2,J}$ in the interior sum we use the variables $r_i=e_i-l_i+1$ so that 
    \bas
    \#&\{\{l_i\}_{i\in K} \ | \ 1\leq l_i\leq e_i, \ \text{and} \ \sum_{i\in K}l_i=(c_{1,I}-1)/2\}\\
    & \= \#\{\{r_i\}_{i\in K} \ | \ 1\leq r_i\leq e_i, \ \text{and} \ \sum_{i\in K}r_i=(c_{2,J}+1)/2\}
    \eas
    and the expression gets rewritten  as
    \be\nonumber
       \mathcal{K}_{cycle} \= \frac{c_{1,I}-1}{2}\sideset{}{^*}\sum_{L\subseteq K} (-1)^{|L|}\binom{(c_{2,J}-1-D_L)/2}{|K|-1}\,,
\ee
where the star expresses again the condition $(c_{2,J}-1-D_L)/2\geq |K|-1$ or equivalently that $c_{2,J}-1-D_L\geq 2(|K|-1)$. We multiply this count by the $(|K|-1)!$ ways to permute the vertices in the cycle, and by the $2^{|K|}$ ways of coloring them, as we removed only the factor~$R$ but according to the coloring rules, vertices in the circle can be joined at vertices black-to-white or vice versa the say clockwise order. We obtain that
    \bas
        \frac{\mathcal{N}(I,J,K)_{h\in I}}{c_{2,J}f_2(m_1,|I|+1)f_2(m_2,|J|+1)} &\=  2^{|K|}(|K|-1)!\mathcal{K}_{cycle}\\
        & \= (c_{1,I}-1)\sum_{L\subseteq K}(-1)^{|L|}[c_{2,J}-1-D_L]_{|K|}\, .
    \eas
where we use the bracket notation for
\be \label{eq:bracket}
[a]_r:=\textbf{1}_{a\geq2(r-1)}f_2(a,r+1).
\ee
Symmetrically, if $h\in J$, the count in terms of $c_{2,J}$ becomes
    \be\nonumber
        \frac{\mathcal{N}(I,J,K)_{h\in J}}{c_{1,I}f_2(m_1,|I|+1)f_2(m_2,|J|+1)}\=(c_{2,J}-1)\sum_{L\subseteq K}(-1)^{|L|}[c_{2,J}-3-D_L]_{|K|}\, .
    \ee
\par
\medskip
\paragraph{\textbf{Case 2: Suppose $h\in K$}} The special vertex in the cycle is now  $G_h^*$, where $2l_h$ is the number of edges facing the $z_2$-face. Let $2l_i-1$ be the number of edges facing the $z_2$-face for every $i\in K/\{h\}$. Then, fixing an order of the vertices of the cycle, the number of attachments is
    \bas
    \mathcal{K}_{cycle} \= & \#\{\{l_i\}_{i\in K} \ | \ 0\leq l_h\leq e_h, \ 1\leq l_i\leq e_i \ \text{for} \ i\in K/\{h\}, \ \sum_{i\in K}l_i=(c_{2,J}-1)/2\}\\
    \= & \#\bigl\{\{l_i\}_{i\in K} \ | \ 1\leq l_i\leq e_i, \ \sum_{i\in K}l_i=(c_{2,J}-1)/2\bigr\}\\
    & + \#\bigl\{\{l_i\}_{i\in K/\{h\}} \ |  \ 1\leq l_i\leq e_i, \text{for} \ i\in K/\{h\}, \ \sum_{i\in K/\{h\}}l_i=(c_{2,J}-1)/2\bigr\}\\
    \= & \sum_{L\subseteq K}(-1)^{|L|}\binom{(c_{2,J}-3-D_L)/2}{|K|-1} + \sum_{L\subseteq K/\{h\}}(-1)^{|L|}\binom{(c_{2,J}-3-D_L)/2}{|K|-2}\\
    \= & \sum_{\substack{L\subseteq K \\ h\in L}}(-1)^{|L|}\binom{(c_{2,J}-3-D_L)/2}{|K|-1} + \sum_{L\subseteq K/\{h\}}(-1)^{|L|}\binom{(c_{2,J}-1-D_L)/2}{|K|-1} 
    \eas
    where the last sum is given by the binomial theorem. Multiplying by the $(|K|-1)!$ ways to permute the vertices, and the $2^{|K|-1}$ ways of coloring the vertices (as the coloring of $G_h^*$ is fixed) the total count of joinings up to rotation is given by 
    \bas
&\phantom{\=} &
\frac{\mathcal{N}(I,J,K)_{h\in K}}{c_{1,I}c_{2,J}f_2(m_1,|I|+1)f_2(m_2,|J|+1)} \=  2^{|K|-1}(|K|-1)!\mathcal{K}_{cycle}\\
&\= & \sum_{\substack{L\subseteq K \\ h\in L}}(-1)^{|L|}[c_{1,I}-3-D_L]_{|K|}  + \sum_{L\subseteq K/\{ h\}}(-1)^{|L|}[c_{1,I}-1-D_L]_{|K|}.
\eas
again with the bracket notation from~\eqref{eq:bracket}.
\par
\medskip
\paragraph{\textbf{Conversion of expressions}} Using the partition notation our goal is to show that $\mathcal{G}_h=\mathcal{F}_h$ where 
\be\nonumber
\mathcal{F}_h \= R \cdot \sum_{\substack{I\sqcup J\sqcup K=\{1,\ldots,h\} \\  c_{1,I},c_{2,J}>0 \\ K'\not= \emptyset}} \mathcal{F}(I,J,K)
\ee
is defined via
\be\nonumber 
\mathcal{F}(I,J,K) \=  c_{1,I}f_2(m_1,|I|+1)c_{2,J}f_2(m_2,|J|+1) \sum_{L\subseteq K}(-1)^{|L|}[c_{2,J}-2-D_L]_{|K|}\, .
\ee
Note that the components of the expression for $\mathcal{G}_h$ depend on the location of the maximal element~$h$, but the expression for $\mathcal{F}_h$ don't. 
We write  $I'\sqcup J' \sqcup K'$ for the partition of $\{1,\ldots,h-1\}$ as above after removing the maximal element~$h$. The computation above gave us to the following expressions:
\bas
\mathcal{N}(I'\cup\{h\},J',K') & \=  (c_{1,I'\cup\{h\}}-1)f_2(m_1,|I'|+2)c_{2,J'}f_2(m_2,|J'|+1)\\
& \quad \cdot\sum_{L\subsetneq K'}(-1)^{|L|}[c_{2,J'}-1-D_{L}]_{|K'|} \\
\mathcal{N}(I',J'\cup\{h\},K') & \=  c_{1,I'}f_2(m_1,|I'|+1)(c_{2,J'\cup\{h\}}-1)f_2(m_2,|J'|+2)\\
& \quad \cdot\sum_{L\subsetneq K'}(-1)^{|L|}[c_{2,J'\cup\{h\}}-3-D_{L}]_{|K'|} \\
\mathcal{N}(I',J',K'\cup\{h\}) & \=  c_{1,I'}f_2(m_1,|I'|+1)c_{2,J'}f_2(m_2,|J'|+1)\\
& \quad \cdot\Bigl(\sum_{L\subseteq K'}(-1)^{|L|}[c_{2,J'}-1-D_L]_{|K'|+1}  \\
& \quad \hspace{0.5cm}  -\sum_{L\subseteq K'}(-1)^{|L|}[c_{2,J'}-3-D_{L\cup\{h\}}]_{|K'|+1}\Bigr) 
\eas
All these functions, as well as $\mathcal{F}(I,J,K)$, involve an auxiliary summation over $L \subset K'$, and thus so does the error function
\be\nonumber
\mathcal{E}(I,J,K) \= \mathcal{F}(I,J,K)-\mathcal{N}(I,J,K)\, 
\ee
where our goal is to show the vanishing of 
\be\nonumber
\frac{\mathcal{F}_h-\mathcal{G}_h}{R} \= \sum_{\substack{I\sqcup J\sqcup K=\{1,\ldots,h\} \\  c_{1,I},c_{2,J}>0 \\ K\not= \emptyset}} \mathcal{E}(I,J,K)\, .
\ee
We will use the shorthand notation
\bas
A \= & c_{2,J'}-2, \qquad \delta^+_r(A) \= & [A]_r-[A+1]_r, \qquad 
\delta^-_r(A) \= & [A]_r-[A-1]_r\,.
\eas
For a fixed partition $I'\sqcup J' \sqcup K'=\{1,\ldots,h-1\}$, there are three types of error functions depending on the location of $h$:
\bas
\mathcal{E}(I'\sqcup \{h\},J',K') \= & c_{2,J'}\,f_2(m_1,|I'|+2)\, f_2(m_2,|J'|+1)\\
&\cdot  \sum_{L\subseteq K}(-1)^{|L|}\left(c_{1,I'\cup\{h\}}\delta^+_{|K'|}(A-D_L) + [A-D_L+1]_{|K'|}\right)\\
\mathcal{E}(I',J'\sqcup \{h\},K') \= & c_{1,I'}f_2(m_1,|I'|+1)f_2(m_2,|J'|+2)\\
&\cdot  \sum_{L\subseteq K}(-1)^{|L|} \left(c_{2,J'\cup\{h\}}\delta^-_{|K'|}(A-D_L-4g_h) + [A-D_L-4g_h-1]_{|K'|}\right)\\
\mathcal{E}(I',J',K'\sqcup \{h\}) \= & c_{1,I'}c_{2,J'}f_2(m_1,|I'|+1)f_2(m_2,|J'|+1)\\
&\cdot  \sum_{L\subseteq K}(-1)^{|L|} \left(\delta^+_{|K'|+1}(A-D_L) - \delta^-_{|K'|+1}(A-D_L - d_h)\right)\, .
\eas
We denote these three functions as $\cE_*(I',J',K')$ for $* = I,J,K$ respectively and denote by $\cE_*(I',J',K')_L$ the summand for fixed~$L$ in each of them. We regroup these sums by defining
\bes
L'=J'\sqcup L, \qquad  H= K' \setminus L \= (I'\sqcup L')^c,
\ees
and hence $K'=H\sqcup L$. Observe that if we fix $L'$, the argument $A-D_{L'\setminus J'}=c_{2,L'}-2+2|L'\setminus J'|$ only depends on the cardinality $|J'|=\colon j$. If we let
\bes
R_*(I',L',H) = \sum_{J \subseteq L'} \cE_*(I',J',K')_{L'}
\ees
the sums become
\bas
\mathcal{R}_I(I',L',H) \= & \sum_{\substack{J'\subseteq L'\\ K'\neq\emptyset}}
(-1)^{|L'\setminus J'|}\,
c_{2,J'}\, f_2(m_1,|I'|+2)\, f_2(m_2,|J'|+1) \cdot
\\
& 
\cdot
\Bigl(c_{1,I'\cup\{h\}}\, \delta^+_{|K'|}
\big(A-D_{L'\setminus J'}\big)
+ \big[A-D_{L'\setminus J'}+1\big]_{|K'|}
\Bigr)\\
\mathcal{R}_J(I',L',H) \= & 
\sum_{\substack{J'\subseteq L'\\ K'\neq\emptyset}}
(-1)^{|L'\setminus J'|}\,
c_{1,I'}\, f_2(m_1,|I'|+1)\, f_2(m_2,|J'|+2) \cdot
\\
&
\cdot \Bigl(
c_{2,J'\cup\{h\}}\, \delta^-_{|K'|} \big(A-D_{L'\setminus J'}-4g_h\big)
+
\big[A-D_{L'\setminus J'}-4g_h-1\big]_{|K'|}
\Bigr)\\
\mathcal{R}_K(I',L',H) \= & \sum_{J'\subseteq L'}
(-1)^{|L'\setminus J'|}\,c_{1,I'}\,c_{2,J'}\,f_2(m_1,|I'|+1)\,f_2(m_2,|J'|+1)\\
&
\cdot\Bigl(\delta^+_{|K'|+1}(A-D_{L'\setminus J'})
- \delta^-_{|K'|+1}(A-D_{L'\setminus J'}-(4g_h-2))\Bigr)\, .
\eas
Using Lemma~\ref{eq:alphaop} these functions reduce to
\bas
\mathcal{R}_I(I',L',H) \= & f_2(m_1,|I'|+2)\Bigl(c_{1,I'\cup\{h\}}\alpha^0_{L',H}(0)-(c_{1,I'\cup\{h\}}-1)\alpha^0_{L',H}(1)\Big)\\ \\
\mathcal{R}_J(I',L',H) \= & c_{1,I'}f_2(m_1,|I'|+1)\Bigl(\alpha^1_{L',H}(2-4g_h)-\alpha^1_{L',H}(1-4g_h)\Bigr)\\ \\
\mathcal{R}_K(I',L',H) \= & c_{1,I'}f_2(m_1,|I'|+1)\Bigl(\alpha^1_{L',H}(0)-\alpha^1_{L',H}(1)\\
& \hspace{3cm} -\alpha^1_{L',H}(2-4g_h)+\alpha^1_{L',H}(1-4g_h) \Bigr)\, .
\eas
As the $\alpha$-operators have the special values
\bas
\alpha^0_{L',H}(0) & \= 0, \quad & \alpha^0_{L',H}(1) & \= -[c_{2,L'}-1]_{|H|}f_2(m_2-c_{2,L'}-1,|L'|+1)\\
\alpha^1_{L',H}(0) & \= 0, \quad & \alpha^1_{L',H}(1) & \= -[c_{2,L'}-1]_{|H|+1}f_2(m_2-c_{2,L'}-1,|L'|+1)
\eas
the sum of the three functions become 
\bas
\mathcal{R}(I',L',H) &\,\coloneq \,  \mathcal{R}_I(I',L',H) + \mathcal{R}_J(I',L',H) + \mathcal{R}_K(I',L',H)\\ 
&\=  f_2(m_1,|I'|+1)[c_{2,L'}-1]_{|H|}f_2(m_2-c_{2,L'}-1,|L'|+1)\\
& \quad \cdot [(m_1-2|I'|+2)(c_{1,I'\cup\{h\}}-1)+c_{1,I'}(c_{2,L'}-2|H|+1)]\, .
\eas
The last step is now to prove the vanishing of
\be\nonumber
\frac{\mathcal{F}_h-\mathcal{G}_h}{R} \= \sum_{I'\sqcup L'\sqcup H=\{1,\ldots,h-1\}} \mathcal{R}(I',L',H)\, .
\ee
For this, we fix $L'$ and let $U=\{1,\ldots,h-1\}\setminus L'$. It is sufficient to show that
\be
\sum_{I'\sqcup H=U} \mathcal{R}(I',L',H) \= 0\, .
\ee
As the common factor $f_2(m_2-c_{2,L'}-1,|L'|+1)$ is independent on $I'$ and $H$, this is equivalent to show the vanishing of the sum 
\bas
S_{L'} &\,\coloneq \,  \sum_{I'\sqcup H=U}  f_2(m_1,|I'|+1)[c_{2,L'}-1]_{|H|}\\
&\hspace{1.5cm} \cdot[(m_1-2|I'|+2)(c_{1,I'\cup\{h\}}-1)+c_{1,I'}(c_{2,L'}-2|H|+1)]\, .
\eas
Here we use the same trick of summing over all $I'\subseteq U$ by first fixing the cardinality of $|I'|=i$. Then the above sum becomes
\bas
S_{L'} \= &\sum_{i=0}^{|U|}\sum_{\substack{I'\subseteq U \\ |I'|=i}}f_2(m_1,i+2)[c_{2,L'}-1]_{|U|-i}(c_{1,I'\cup\{h\}}-1)\\
&\hspace{1.5cm} + f_2(m_1,i+1)[c_{2,L'}-1]_{|U|-i+1}c_{1,I'}\, .
\eas
By using the identities 
\bas
\sum_{\substack{I'\subseteq U \\ |I'|=i}}c_{1,I'} &\=  \binom{|U|-1}{i}(m_1-2i+2)
 + \binom{|U|-1}{i-1}(c_{1,U}+2|U|-2i) \\ 
\sum_{\substack{I'\subseteq U \\ |I'|=i}}(c_{1,I'\cup\{h\}}-1) &\= \sum_{\substack{I'\subseteq U \\ |I'|=i}}c_{1,I'}  -(4g_h+1)\binom{|U|}{i}
\eas
we obtain
\bas
S_{L'} &\= \sum_{i=0}^{|U|}\,\, f_2(m_1,i+2)[c_{2,L'}-1]_{|U|-i} \cdot\\
& \cdot \Bigg( \binom{|U|-1}{i}(m_1-2i+2) + \binom{|U|-1}{i-1}(c_{1,U}+2|U|-2i) -(4g_h+1)\binom{|U|}{i}\Bigg) \\
&\hspace{0.7cm} + f_2(m_1,i+1)[c_{2,L'}-1]_{|U|-i+1} \, \cdot \\
& \cdot \Bigg( \binom{|U|-1}{i}(m_1-2i+2) + \binom{|U|-1}{i-1}(c_{1,U}+2|U|-2i)\Bigg)\, .
\eas
We substitude $c_{1,U}\= 4g_h-c_{2,L'}$ and $\binom{|U|}{i}\= \binom{|U|-1}{i} + \binom{|U|-1}{i-1}$ and
merge the terms with the common binomial coefficient and conclude that $S_{L'}=0$.
\end{proof}
\par
\begin{lemma} \label{eq:alphaop}
The operator 
\bas
\alpha_{L',H}^{\epsilon}(u) \, :=\, \sum_{J'\subseteq L'} (-1)^{|L'\setminus J'|} c_{2,J'}f_2(m_2,|J'|+1)[A-D_{L'\setminus J'}+u]_{|K'|+ \epsilon}\, .
\eas
has the simpler expression
\be\nonumber
    \alpha_{L',H}^{\epsilon}(u) \= -u\,[c_{2,L'}+u-2]_{|H|+\epsilon}\,f_2(m_2-c_{2,L'}-u,|L'|+1)\, .
\ee
\end{lemma}
\par
\begin{proof} By fixing $|J'|=j$, this operator becomes
\bas
\alpha_{L',H}^{\epsilon}(u) \= \sum_{j=0}^{|L'|}\alpha_{L',H}^{\epsilon}(u;j).
\eas
with
\bas
\alpha_{L',H}^{\epsilon}(u;j) \,\coloneq\, (-1)^{|L'|-j}f_2(m_2,j+1)[A-D_{L'\setminus J'}+u]_{|K'|+ \epsilon}\sum_{\substack{J'\subseteq L' \\ |J'|=j}}c_{2,J'}\, .
\eas
Let $r=|L'|-j$. Using the identities
\bas
\sum_{\substack{J'\subseteq L' \\ |J'|=j}}c_{2,J'} &\=  \binom{|L'|-1}{j}(m_2+2-2j) + \binom{|L'|-1}{j-1}(c_{2,L'}+2r)\\
[A-D_{L'\setminus J'}+u]_{|K'|+\epsilon} &\=  [c_{2,L'}+u-2]_{|H|+\epsilon}f_2(c_{2,L'}+2(r-1)+u,r+2)
\eas
and the Vandermonde convolution identity (with $2$-step factorials instead of the usual factorials)
\bas
\sum_{r=0}^{|L'|-1}&(-1)^r\binom{|L'|-1}{r}f_2(m_2,|L'|+1-r)f_2(c_{2,L'}+u+2r-2,r+2) \\
& \= f_2(m_2-c_{2,L'}-u,|L'|+1)
\eas
the claim follows.
\end{proof}
%
%

%% file: sec_examples.tex
\section{Examples and outlook} \label{sec:examples}

In the introduction we presented an example illustrating our main theorem that the DR-completed volume agrees with the ribbon graph completed volume for quadratic differentials, i.e.\ for $k=2$. In this section we present another example for abelian differentials illustrating Theorem~\ref{thm:DRcompletedvolgraphs} that gives the DR-completed volume as a graph sum of top tautological intersection numbers.
\par
Moreover we give an outlook comparing the two possible ways to compute completed volumes by evaluations in the tautological ring.
\par
\medskip
\paragraph{\textbf{The stratum $\mu=(4,2,-2)$ with $k=1$}} Since this stratum has three zeros there are no special star graphs and the possible list of sunflower graphs for this stratum is as follows:
\bas
D_1 &\=\left[\begin{tikzpicture}[
  baseline={([yshift=-.5ex]current bounding box.center)},
  scale=1.5,
  very thick,
  bend angle=30,
  comp/.style={circle,draw,thin,inner sep=2pt,font=\tiny},
  order bottom left/.style={pos=.05,left,font=\tiny},
  order top left/.style={pos=.9,left,font=\tiny},
  order bottom right/.style={pos=.05,right,font=\tiny},
  order top right/.style={pos=.9,right,font=\tiny},
  order node dis/.style={text width=.75cm},
]

\node[comp] (A) at (0,0) {1};
\node[comp] (B) at (0.5,0) {2};

\node[comp,fill] (E) at (0.25,-0.75) {};

\node[order node dis,above left] (A-top) at (A.north east) {$-2$};
\path (A) edge [shorten >=4pt] (A-top.center);


\node [minimum width=18pt,above right] (A-top) at (A.north east) {$2$};
\path (A) edge [shorten >=5pt] (A-top.center);

\node[minimum width=18pt,below left] (E-half) at (E.south east) {$4$};
\path (E) edge [shorten >=4pt] (E-half.center);

\draw (A) edge node[order bottom left,yshift=-1pt] {$0$} node[order top left] {$-2$} (E);
\draw (B) edge node[order bottom left] {$2$} node[order top right] {$-4$} (E);

\end{tikzpicture}\right]
\qquad
D_2\=\left[\begin{tikzpicture}[
  baseline={([yshift=-.5ex]current bounding box.center)},
  scale=1.5,
  very thick,
  bend angle=30,
  comp/.style={circle,draw,thin,inner sep=2pt,font=\tiny},
  order bottom left/.style={pos=.05,left,font=\tiny},
  order top left/.style={pos=.9,left,font=\tiny},
  order bottom right/.style={pos=.05,right,font=\tiny},
  order top right/.style={pos=.9,right,font=\tiny},
  order node dis/.style={text width=.75cm},
]

\node[comp] (A) at (0,0) {1};
\node[comp] (B) at (0.75,0) {1};
\node[comp] (C) at (1.25,0) {1};

\node[comp,fill] (E) at (0.5,-0.75) {};

\node[order node dis,above left] (A-top) at (A.north east) {$-2$};
\path (A) edge [shorten >=4pt] (A-top.center);


\node [minimum width=18pt,above right] (A-top) at (A.north east) {$2$};
\path (A) edge [shorten >=5pt] (A-top.center);

\node[minimum width=18pt,below left] (E-half) at (E.south east) {$4$};
\path (E) edge [shorten >=4pt] (E-half.center);

\draw (A) edge node[order bottom left,yshift=-2pt] {$0$} node[order top left,yshift=-2pt] {$-2$} (E);
\draw (B) edge node[order bottom left] {$0$} node[order top right,yshift=-2pt] {$-2$} (E);
\draw (C) edge node[order bottom left] {$0$} node[order top left,yshift=8pt] {$-2$} (E);

\end{tikzpicture}\right] 
D_3 \=\left[\begin{tikzpicture}[
  baseline={([yshift=-.5ex]current bounding box.center)},
  scale=1.5,
  very thick,
  bend angle=30,
  comp/.style={circle,draw,thin,inner sep=2pt,font=\tiny},
  order bottom left/.style={pos=.05,left,font=\tiny},
  order top left/.style={pos=.9,left,font=\tiny},
  order bottom right/.style={pos=.05,right,font=\tiny},
  order top right/.style={pos=.9,right,font=\tiny},
  order node dis/.style={text width=.75cm},
]

\node[comp] (A) at (0,0) {2};
\node[comp] (B) at (0.5,0) {1};

\node[comp,fill] (E) at (0.25,-0.75) {};

\node[order node dis,above left] (A-top) at (A.north east) {$-2$};
\path (A) edge [shorten >=4pt] (A-top.center);


\node [minimum width=18pt,above right] (A-top) at (A.north east) {$2$};
\path (A) edge [shorten >=5pt] (A-top.center);

\node[minimum width=18pt,below left] (E-half) at (E.south east) {$4$};
\path (E) edge [shorten >=4pt] (E-half.center);

\draw (A) edge node[order bottom left,yshift=-1pt] {$2$} node[order top left] {$-4$} (E);
\draw (B) edge node[order bottom left] {$0$} node[order top right] {$-2$} (E);

\end{tikzpicture}\right] \\
D_4 &\=\left[\begin{tikzpicture}[
  baseline={([yshift=-.5ex]current bounding box.center)},
  scale=1.5,
  very thick,
  bend angle=30,
  comp/.style={circle,draw,thin,inner sep=2pt,font=\tiny},
  order bottom left/.style={pos=.05,left,font=\tiny},
  order top left/.style={pos=.9,left,font=\tiny},
  order bottom right/.style={pos=.05,right,font=\tiny},
  order top right/.style={pos=.9,right,font=\tiny},
  order node dis/.style={text width=.75cm},
]

\node[comp] (A) at (0,0) {2};
\node[comp] (B) at (0.5,0) {1};

\node[comp,fill] (E) at (0.25,-0.75) {};

\node[order node dis,above left] (A-top) at (A.north east) {$-2$};
\path (A) edge [shorten >=4pt] (A-top.center);


\node [minimum width=18pt,above right] (A-top) at (A.north east) {$4$};
\path (A) edge [shorten >=5pt] (A-top.center);

\node[minimum width=18pt,below left] (E-half) at (E.south east) {$2$};
\path (E) edge [shorten >=4pt] (E-half.center);

\draw (A) edge node[order bottom left,yshift=-1pt] {$0$} node[order top left] {$-2$} (E);
\draw (B) edge node[order bottom left] {$0$} node[order top right] {$-2$} (E);

\end{tikzpicture}\right] \qquad \qquad 
D_5 \= \left[\begin{tikzpicture}[
  baseline={([yshift=-.5ex]current bounding box.center)},
  scale=1.5,
  very thick,
  bend angle=30,
  comp/.style={circle,draw,thin,inner sep=2pt,font=\tiny},
  order bottom left/.style={pos=.05,left,font=\tiny},
  order top left/.style={pos=.9,left,font=\tiny},
  order bottom right/.style={pos=.05,right,font=\tiny},
  order top right/.style={pos=.9,right,font=\tiny},
  order node dis/.style={text width=.75cm},
]

\node[comp,fill] (A) at (0,0) {};
\node[comp] (B) at (0.5,0) {1};
\node[comp] (C) at (1.25,0) {2};

\node[comp,fill] (E) at (0.25,-0.75) {};
\node[comp,fill] (F) at (1,-0.75) {};

\node[order node dis,above left] (A-top) at (A.north east) {$-2$};
\path (A) edge [shorten >=4pt] (A-top.center);

\node[minimum width=18pt,below left] (E-half) at (E.south east) {$2$};
\path (E) edge [shorten >=4pt] (E-half.center);

\node[minimum width=18pt,below left] (F-half) at (F.south west) {$4$};
\path (F) edge [shorten >=4pt] (F-half.center);

\draw (A) edge node[order bottom left] {$0$} node[order top left] {$-2$} (E);
\draw (B) edge node[order bottom left] {$0$} node[order top right] {$-2$} (E);
\draw (C) edge node[order bottom right,yshift=-2pt] {$2$} node[order top right] {$-4$} (F);
\draw (A) edge node[order bottom right, yshift=2pt] {$0$} node[order top left,yshift=-1pt] {$-2$} (F);

\end{tikzpicture}\right] \\
D_6 &\=\left[\begin{tikzpicture}[
  baseline={([yshift=-.5ex]current bounding box.center)},
  scale=1.5,
  very thick,
  bend angle=30,
  comp/.style={circle,draw,thin,inner sep=2pt,font=\tiny},
  order bottom left/.style={pos=.05,left,font=\tiny},
  order top left/.style={pos=.9,left,font=\tiny},
  order bottom right/.style={pos=.05,right,font=\tiny},
  order top right/.style={pos=.9,right,font=\tiny},
  order node dis/.style={text width=.75cm},
]

\node[comp,fill] (A) at (0,0) {};
\node[comp] (B) at (0.5,0) {1};
\node[comp] (C) at (1.25,0) {1};
\node[comp] (D) at (1.75,0) {1};

\node[comp,fill] (E) at (0.25,-0.75) {};
\node[comp,fill] (F) at (1,-0.75) {};

\node[order node dis,above left] (A-top) at (A.north east) {$-2$};
\path (A) edge [shorten >=4pt] (A-top.center);

\node[minimum width=18pt,below left] (E-half) at (E.south east) {$2$};
\path (E) edge [shorten >=4pt] (E-half.center);

\node[minimum width=18pt,below left] (F-half) at (F.south west) {$4$};
\path (F) edge [shorten >=4pt] (F-half.center);

\draw (A) edge node[order bottom left] {$0$} node[order top left] {$-2$} (E);
\draw (A) edge node[order bottom right, yshift=2pt] {$0$} node[order top left,yshift=-1pt] {$-2$} (F);
\draw (B) edge node[order bottom left] {$0$} node[order top right] {$-2$} (E);
\draw (C) edge node[order bottom right,yshift=-2pt] {$0$} node[order top left, yshift=5pt] {$-2$} (F);
\draw (D) edge node[order bottom right,yshift=-2pt] {$0$} node[order top right] {$-2$} (F);

\end{tikzpicture}\right] \qquad 
D_7 \=\left[\begin{tikzpicture}[
  baseline={([yshift=-.5ex]current bounding box.center)},
  scale=1.5,
  very thick,
  bend angle=30,
  comp/.style={circle,draw,thin,inner sep=2pt,font=\tiny},
  order bottom left/.style={pos=.05,left,font=\tiny},
  order top left/.style={pos=.9,left,font=\tiny},
  order bottom right/.style={pos=.05,right,font=\tiny},
  order top right/.style={pos=.9,right,font=\tiny},
  order node dis/.style={text width=.75cm},
]

\node[comp] (A) at (0,0) {1};
\node[comp] (B) at (0.5,0) {1};
\node[comp] (C) at (1.25,0) {1};

\node[comp,fill] (E) at (0.25,-0.75) {};
\node[comp,fill] (F) at (1,-0.75) {};

\node[order node dis,above left] (A-top) at (A.north east) {$-2$};
\path (A) edge [shorten >=4pt] (A-top.center);

\node[minimum width=18pt,below left] (E-half) at (E.south east) {$2$};
\path (E) edge [shorten >=4pt] (E-half.center);

\node[minimum width=18pt,below left] (F-half) at (F.south west) {$4$};
\path (F) edge [shorten >=4pt] (F-half.center);

\draw (A) edge node[order bottom left] {$0$} node[order top left] {$-2$} (E);
\draw (B) edge node[order bottom left] {$0$} node[order top right] {$-2$} (E);
\draw (C) edge node[order bottom right,yshift=-2pt] {$0$} node[order top right] {$-2$} (F);
\draw (A) edge node[order bottom right, yshift=2pt] {$2$} node[order top left,yshift=-1pt] {$-4$} (F);

  \end{tikzpicture}\right]
\eas
Their volume contributions, as computed using the sage-package \texttt{diffstrata} are shown in the following table: 
\begin{center}
  \begin{tabular}{|c|c|c|c|c|}
  \hline
  \rule{0pt}{14pt} $\Gamma$ & $\frac{1}{|Aut(\Gamma)|k^{h_{\ab}}}$ & $\prod_{e\in E(\Gamma)}\kappa_{e}$ & $\prod_{v\in V(\Gamma)}\text{vol}(\mu_v)$ & $\text{vol}({\Gamma})$ \\ \hline
  $\bullet$ & 1 & 1 & 23/9216 & 23/9216 \\ \hline
  $D_1$ & 1 & 3 & 1/5120 & 3/5120 \\ \hline
  $D_2$ & 1/2 & 1 & 1/1152 & 1/2304 \\ \hline
  $D_3$ & 1 & 3 & -1/1536 & -1/512  \\ \hline
  $D_4$ & 1 & 1 & -23/27648 & -23/27648 \\ \hline
  $D_5$ & 1 & 3 & -1/15360 & -1/5120  \\ \hline
  $D_6$ & 1/2 & 1 & -1/3456 & -1/6912 \\ \hline
  $D_7$ & 1 & 3 & 1/4608 & 1/1536 \\ \hline
\end{tabular}  
\end{center}
These contributions add up to
\be
\ol{\vol}^\DR(\mu)\= 1/960 
\ee
which is confirmed by the \texttt{DR\_cycle}-method (with \texttt{chiodo\_coeff=True}) of the sage-package \texttt{admcycles}.  
\par
\par
\medskip
\paragraph{\textbf{Two expresssions in the tautological ring}} Our main Theorem~\ref{thm:intromain} shows that the completed volume can be computed by two quite different expressions in the tautological ring that we briefly compare.
\par
First, the Ribbon graph completed volume for $k=2$ and~$\mu$ with only odd entries is computed as follows as a sum over all stable graphs~$\Delta$ as in Section~\ref{sec:ribbon} with edge factors and vertex factors as follows. Let $\bfb = (b_e)_{e \in E(\Delta)}$ be a multi-index and $\bfd_v$ be a multi-index 
\be
\ol{\Vol}(\mu) \= \sum_\Delta \mathcal{Z}\Bigl(\prod_{e \in E(\Delta)} b_e \cdot \prod_{v \in V(\Delta)} \sum_{|\bfd_v| = \dim(\mu_v)} 2^{g-\dim(\mu_v)} \langle \tau_\bfd \rangle_{\mu_v}\frac{\bfb^{\bfd_v}}{\bfd !} \Bigr)
\ee
using the linear operator $\mathcal{Z}: \bfb^\bfd \mapsto \prod d_i! \zeta(d_i+1)/(\sum (d_i+1))!$ and where
\be
\langle \tau_\bfd \rangle_{\bfm} \= \int_{W_{\bfm,n}} \prod_{i \geq 0} \psi_i^{d_i}, \qquad
\bfd = (d_1,\ldots,d_n)
\ee
is the integral over Witten-Kontsevich's combinatorial classes $W_{\bfm,n} \subset \moduli[g,n]$, whose dimension we denoted by $\dim(\bfm)$,  see \cite[Section~2.1]{DGY}, \cite{Witten2dgrav,KontAiry}.
\par
Second, the DR-completed volume can be evaluated using the Chiodo-class expression in Theorem~\ref{thm:ChiodoHolmes} and~\eqref{eq:ChClass}. It is also a sum over stable graphs, but in its present form it is not very efficient to evaluate. One reason for this is the dependence on the auxiliary weighting parameter~$r$. We conjecture that the Chiodo-coefficient needed here, of top codimension is actually a monomial in~$r$, which already gives significant computational simplification. Moreover, the DR-completed volume is defined for any~$k$ and any~$\mu$ and exhibits interesting piecewise polynomial behaviour that further improves evaluation. We plan to address this in a sequel to the paper and hope that at least one of the two expressions is suitable to approach the large genus limit conjecture \cite{ADGZZconj}.